\documentclass[12pt]{article}
\usepackage[utf8]{inputenc}
\usepackage[english]{babel}
\usepackage{comment}
\usepackage{tikz}
\usepackage[dvipsnames]{xcolor}
\usepackage{arxiv}

\definecolor{myorange}{RGB}{255, 176, 1}
\definecolor{mygreen}{RGB}{55, 184, 78}
\definecolor{mycyan}{RGB}{51, 255, 255}

\usepackage{amsfonts}
\usepackage{amsthm,amsmath,amssymb}

\usepackage{authblk}

\usepackage{amsthm}

\newtheorem{theorem}{Theorem}[section]
\newtheorem{lemma}[theorem]{Lemma}

\newtheorem{corollary}[theorem]{Corollary}

\newtheorem{definition}[theorem]{Definition}

\numberwithin{equation}{section}

\date{}

\usepackage{graphicx}
\usepackage{caption}
\usepackage{subcaption}

\begin{document}

\title{Lipschitz Stability in the Simultaneous Determination of Polygonal Inclusions and Constant Conductivities}
\author[1]{Dai Tianrui}

\affil[1]{Department of Mathematics and Computer Science, Università degli Studi di Firenze}

\maketitle

\begin{abstract}
This work establishes a Lipschitz stability result for identifying unknown polygonal inclusions along with their unknown constant conductivity values, given boundary measurements encoded in the Dirichlet-to-Neumann map. We also obtain Lipschitz estimates for the Hausdorff distance between inclusion boundaries and for the difference of the conductivity values.
\end{abstract}

\section{Introduction}
   We consider the following conductivity problem with Dirichlet boundary condition:
   \begin{equation}\label{original_model}
       \begin{cases}
           -\text{div} \left(\gamma \nabla u\right)=0 \text{ in } \Omega\\
           u=f \text{ on }\partial \Omega
       \end{cases},
   \end{equation}
where the solution $u$ represents an electric potential in a dielectric $\Omega$. The function $f$ represents the boundary potential. Assume there is an inclusion $D\subset\Omega$ of other conducting material such that the spatial conductivity distribution $\gamma$ could be written as  
\begin{equation}
    \gamma=\begin{cases}
        k \text{ in } D\\
        1 \text{ in } \Omega\setminus\overline{D}
    \end{cases}.
\end{equation}
The inverse problem that we are interested is recovering the conductivity $\gamma$ using boundary measurements. Such inverse problems arise naturally in various applications across physics, engineering, and medical imaging, particularly those involving heterogeneous media.  

The related study of this kind of inverse problem with this specific form of parameters has a long story. In the degenerate case, where the unknown inclusion D is taken to be insulating, i.e., when $k=0$, the Cauchy data of $u$ on $\partial\Omega$ uniquely determine the inclusion has been shown by \cite{beretta1998stable}. Furthermore, when the conductivity in the inclusion is a known nonzero constant different from one, \cite{friedman1989uniqueness} showed that a convex polygonal inclusion can be uniquely determined by one Cauchy measurement. In 1995, \cite{seo1995uniqueness} proved that the convexity assumption on the polygon can be dropped provided two appropriate probing currents. 

Meanwhile, the stability study of this kind of inverse problem also has been studied for a long time. Using the prior information of such a special form of conductivity coefficients, compared to the typical logarithmic stability result, it usually presents a better stability behavior. In \cite{alessandrini2005lipschitz}, the authors showed that if all Cauchy pairs are known, the values of a piecewise constant conductivity $\gamma$ respect to a known partition of $\Omega$ Lipschitz continuously depend on these data. Later, finite number of Cauchy pairs is sufficient for this result are shown in \cite{harrach2019uniqueness},\cite{alberti2019calderon} and \cite{alberti2022infinite}. The inverse problem with the unknown position of vertices for the polygon inclusion $D$ with known conductivity constant $k$ is treated in \cite{BerFra20}. The author showed that such problem has a Lipschitz stability in terms of the Hausdorff distance between the boundaries of two admissible inclusions. Inspired by their work, we work on the case with a relaxation of the condition that the conductivity constant $k$ is known.

The structure of this paper is organized as follows: 
\begin{itemize}
    \item In section 2, we introduce the basic assumptions concerning the interested domain $\Omega$, the unknown inclusion $D$ (which will be denoted as $P$ in the sequel) and the conductivity constant $k$. We also state our main stability result that the conductivity $\gamma$ exhibits Lipschitz stability with respect to the corresponding Dirichlet-to-Neumann (DtN) map.
    \item In section 3, we establish a geometric result, demonstrating that the DtN map uniquely determines the number of vertices of the unknown polygonal inclusion.
    \item In section 4, utilizing the geometric findings from Section 3, we show that the difference between two conductivities can be explicitly characterized by the variations of the polygon's vertices and the shift in the conductivity constants. Subsequently, we derive the exact formulation of the corresponding Fréchet derivative and establish its upper bound.
    \item In sections 5 and 6, we conclude the paper with the proof of the main stability theorem, which critically relies on establishing a lower bound for the associated Fréchet derivative
\end{itemize}

\section{Assumptions and main result}
    In this section, we introduce some basic assumptions on the target domain $\Omega$, the polygonal inclusion $P$, and the corresponding conductivity $\gamma$. All parameters appearing in the assumptions will be considered as a prior information. Using such prior information, we then present our main Lipschitz stability result.

    The target domain we considered will be denoted as $\Omega$, which is a bounded domain in $\mathbb{R}^{2}$. We assume that there exists some constant $L>0$ such that diam$\left(\Omega\right)\leq L$. 
    In addition, we assume that $\partial\Omega$ is Lipschitz with constants $\left(r_{0},K_{0}\right)$, that is, for every point $P\in\partial\Omega$ there exists a coordinate system with $P=0$ in which, writing $Q:=\left[-r_{0},r_{0}\right]\times\left[-K_{0}r_{0},K_{0}r_{0}\right]$,
    \begin{equation*}
        \Omega\cap Q=\left\{\left(x_{1},x_{2}\right)\in Q:\ x_{2}>\phi\left(x_{1}\right)\right\},
    \end{equation*}
    for a Lipschitz continuous function $\phi$ with Lipschitz constant at most $K_{0}$. In particular $\partial\Omega\cap Q$ is the graph $\left\{\left(x_{1},\phi\left(x_{1}\right)\right):x_{1}\in\left[-r_{0},r_{0}\right]\right\}$.

    We now introduce the set of polygonal inclusions:
    
    \begin{definition}\label{def1}
        Let $N_{0}\in\mathbb{N}$, $0<\alpha_{0}<\frac{\pi}{2}$, $d_{0}>0$, and $\left(r_{0},K_{0}\right)$
        be the same Lipschitz constants for $\partial\Omega$. 
        In $\mathbb{R}^{2}$, a set of polygons $\mathcal{A}$ is called \emph{a nice inclusions set} with
        the prior information $\left(N_{0},\alpha_{0},d_{0},r_{0},K_{0}\right)$, if for any $\mathcal{P}\in\mathcal{A}$, there holds
        \begin{enumerate}
            \item 
                $\mathcal{P\subset}\Omega$ is a closed, simply connected, simple
                polygon; in particular $\mathcal P=\overline{\operatorname{int}\mathcal P}$ and $\partial\mathcal P$ is a finite polygonal line. $\partial\mathcal{P}$ is Lipschitz with constants $\left(r_{0},K_{0}\right)$, meaning that at every point of $\partial\mathcal P$ the graph representation used for $\partial\Omega$ above holds for $\operatorname{int}\mathcal P$ in place of $\Omega$: for a chart rectangle $Q$ centered at the point, $\left(\operatorname{int}\mathcal P\right)\cap Q=\left\{\left(x_{1},x_{2}\right)\in Q:\ x_{2}>\phi\left(x_{1}\right)\right\}$ with $\operatorname{Lip}\phi\le K_{0}$, and $\partial\mathcal P\cap Q=\left\{\left(x_{1},\phi\left(x_{1}\right)\right):x_{1}\in\left[-r_{0},r_{0}\right]\right\}$.
            \item 
                Denote $^{\#}\mathcal{P}$ as the number of vertices of $\mathcal{P}$,
                then $^{\#}\mathcal{P}\leq N_{0}$. Denote $x_{i}$ as the $i$-th
                vertex of $\mathcal{P}$ and denote $\alpha_{i}$ as the angle in $x_{i}$, then $\alpha_{0}\leq\alpha_{i}\leq2\pi-\alpha_{0}$ and $\left|\alpha_{i}-\pi\right|\geq\alpha_{0}$.
            \item  
                Every side of $\mathcal{P}$ has length greater than $d_{0}$ and dist$\left(\mathcal{P},\partial\Omega\right)$
                (the distance between $\mathcal{P}$ and $\partial\Omega$) is greater than $d_{0}$.
        \end{enumerate}
    \end{definition}

    The vertices of every polygon $\mathcal P\in\mathcal A$ are labeled counterclockwise. When two polygons with the same number of vertices are compared, the labels are chosen according to the vertex correspondence of Lemma \ref{same_number_result}, up to a cyclic permutation fixed once and for all.
    
    We now introduce the following conductivity set:

    \begin{definition}\label{def2}
        Let $\lambda_{0}>1$ and $\lambda_{1}>1$, we say $\mathcal{K}$ is \emph{a nice conductivity set} with the prior information $\left(\lambda_{0},\lambda_{1}\right)$, if for any $k\in\mathcal{K}$, there holds $\frac{1}{\lambda_{0}}<k<\lambda_{0}$ and $\lambda_{1}>\left|k-1\right|>\frac{1}{\lambda_{1}}$. 
    \end{definition}

    The coefficients appearing in the Definition $\ref{def1}$ and Definition $\ref{def2}$ will be given as a prior information. The domain $\Omega$ itself is fixed throughout the paper, and the estimates below are carried out on $\Omega$ and on the extension domain $\Omega_{0}$ introduced in Section 5, the fixed ball $B_{2L}\left(x_c\right)$, which is determined by $\Omega$ and $L$; we therefore include $\Omega$ in the prior information. We denote them by:
    \[
        PI:=\left(\Omega,N_{0},\alpha_{0},d_{0},r_{0},K_{0},\lambda_{0},\lambda_{1},L\right)
    \]
    and every constant appearing below is understood to depend only on these data.

    For any polygon $\mathcal{P}\in\mathcal{A}$ and any $k\in\mathcal{K}$, we define a piecewise conductivity in $\Omega$: 
    \[
        \gamma_{k}^{\mathcal{P}} := \chi_{\mathcal{P}}\left(k-1\right)+1,
    \]
    where $\chi_{\mathcal{P}}$ is the characteristic function of the inclusion $\mathcal{P}$.

    The corresponding Dirichlet-to-Neumann map associated with such conductivity $\gamma_{k}^{\mathcal{P}}$ is then defined as 
    \begin{align*}
        \Lambda_{k}^{\mathcal{P}}:H^{1/2}\left(\partial\Omega\right)&\rightarrow H^{-1/2}\left(\partial\Omega\right)\\
        f&\mapsto\Lambda_{k}^{\mathcal{P}}\left(f\right)=\gamma_{k}^{\mathcal{P}}\,\partial_{\nu}u_{f}=\partial_{\nu}u_{f},
    \end{align*}

    where $u_{f}$ is the solution of the following Dirichlet boundary problem:
    \[
        \begin{cases}
            -\text{div}\left(\gamma_{k}^{\mathcal{P}}\nabla u_{f}\right)=0 & \text{in }\Omega\\
            u_{f}=f & \text{on }\partial\Omega,
        \end{cases}
    \]

    here $\nu$ represents the outer unit normal vector on $\partial\Omega$; the equality $\gamma_{k}^{\mathcal{P}}\partial_{\nu}u_{f}=\partial_{\nu}u_{f}$ holds because $\operatorname{dist}\left(\mathcal P,\partial\Omega\right)\ge d_{0}>0$ implies $\gamma_{k}^{\mathcal P}\equiv1$ in a neighborhood of $\partial\Omega$. With this convention the Dirichlet-to-Neumann map $\Lambda_{k}^{\mathcal{P}}$'s operator norm in the space of linear operators space $\mathcal{L}\left(H^{1/2}\left(\partial\Omega\right);H^{-1/2}\left(\partial\Omega\right)\right)$ is defined by 
    
    \[
        \left\Vert \Lambda_{k}^{\mathcal{P}}\right\Vert _{*}:=
        \text{sup}
        \left(
            \frac{\left\Vert \Lambda_{k}^{\mathcal{P}}\left(f\right)\right\Vert _{H^{-1/2}\left(\partial\Omega\right)}}
                {\left\Vert f\right\Vert _{H^{1/2}\left(\partial\Omega\right)}}
            :f\neq0
        \right).
    \]

    The main aim of this paper is to study the Lipschitz stability of the conductivity $\gamma_{k}^{\mathcal{P}}$ corresponding to the Dirichlet-to-Neumann map $\Lambda_{k}^{\mathcal{P}}$. Specifically, we have the following main result:
    
    \begin{theorem}\label{mainthm}
        There exists some constant $C$ only depends on the prior information PI, such that for any $\mathcal{P}_{1},\mathcal{P}_{2}\in\mathcal{A}$ and $k_{1},k_{2}\in\mathcal{K}$, we have
        \begin{equation}
            \left\Vert \gamma_{k_{1}}^{\mathcal{P}_{1}}-\gamma_{k_{2}}^{\mathcal{P}_{2}}\right\Vert _{L^{1}\left(\Omega\right)}\leq C\left\Vert \Lambda_{k_{1}}^{\mathcal{P}_{1}}-\Lambda_{k_{2}}^{\mathcal{P}_{2}}\right\Vert _{*}.
        \end{equation}
    \end{theorem}

\section{Some results on distance between polygons}

    The key point to prove Theorem \ref{mainthm} is to describe the change between two conductivities $\gamma_{k_{1}}^{\mathcal{P}_{1}}$ and $\gamma_{k_{2}}^{\mathcal{P}_{2}}$, which will be decomposed into the transformation between two different polygons $\mathcal{P}_{1}$ and $\mathcal{P}_{2}$ as well as the change between two different unknown constants $k_{1}$ and $k_{2}$. The latter one could be easily formulated by $\lvert k_{1}-k_{2}\rvert$. For the variation in shapes, a group of vector fields defined on the vertices of polygons will help us to formulate the geometric deformation. This naturally requires such two polygons $\mathcal{P}_{1}$ and $\mathcal{P}_{2}$ have the same number of vertices. In fact, thanks to Proposition 3.4 in \cite{BerFra20}, if the difference between two Dirichlet-to-Neumann maps is small enough in the sense of the operator norm, then not only the polygons have the same number of vertices, but also the distances between vertices are small enough. A similar discussion can also be applied for unknown conductivity constants in this paper. We will present some related results here.
    
    By Lemma 2.2 in \cite{BdeHFS} and the elementary estimate below, the admissible conductivities are uniformly bounded at each fixed fractional Sobolev order:

    \begin{lemma} \label{holder_space_result}
        Fix $s_{*}:=\frac14$. For any polygon $\mathcal{P}\in\mathcal{A}$
        and any conductivity constant $k\in\mathcal{K}$, there exists some constant $\Gamma_{0}$ which only depends on the prior information $PI$ such that the conductivity associated with $\mathcal{P}$ and $k$, namely $\gamma_{k}^{\mathcal{P}}$ as defined in Definition \ref{def2}, satisfies
    \[
        \left\Vert \gamma_{k}^{\mathcal{P}}\right\Vert _{H^{s_{*}}\left(\Omega\right)}\leq\Gamma_{0}.
    \]
    \end{lemma}
    \begin{proof}
        Since $\gamma_{k}^{\mathcal{P}}=1+\left(k-1\right)\chi_{\mathcal P}$ with $\left\lvert k-1\right\rvert\le\lambda_{1}$, it suffices to bound $\chi_{\mathcal P}$ uniformly in $H^{s_{*}}\left(\mathbb R^{2}\right)$. For every $z\in\mathbb R^{2}$,
        \[
            \left\lVert\chi_{\mathcal P}\left(\cdot+z\right)-\chi_{\mathcal P}\right\rVert_{L^{2}\left(\mathbb R^{2}\right)}^{2}
            =\left\lvert\mathcal P\triangle\left(\mathcal P-z\right)\right\rvert
            \le\min\left\{2\left\lvert\mathcal P\right\rvert,\ \mathrm{Per}\left(\mathcal P\right)\left\lvert z\right\rvert\right\},
        \]
        where $\mathrm{Per}\left(\mathcal P\right)$ denotes the perimeter. Splitting the double integral in the Slobodeckij seminorm at $\left\lvert z\right\rvert=1$ gives
        \[
            \left[\chi_{\mathcal P}\right]_{H^{s}\left(\mathbb R^{2}\right)}^{2}
            \le C\left(\frac{\mathrm{Per}\left(\mathcal P\right)}{1-2s}+\frac{\left\lvert\mathcal P\right\rvert}{s}\right),
            \qquad 0<s<\frac12 .
        \]
        Every side of $\mathcal P$ has length at most $L$ (the diameter bound of $\Omega$), hence $\mathrm{Per}\left(\mathcal P\right)\le N_{0}L$, and $\left\lvert\mathcal P\right\rvert\le\left\lvert\Omega\right\rvert\le\pi L^{2}$; together with $\left\lVert\chi_{\mathcal P}\right\rVert_{L^{2}}\le\left\lvert\Omega\right\rvert^{1/2}$ this proves the claim at $s=s_{*}=\frac14$ with a constant depending only on $PI$.
    \end{proof}

    Then, combined Lemma \ref{holder_space_result} with Theorem 1.1 in \cite{clop2010stability}, we have the following logarithmic stability estimate:
    
    \begin{lemma}\label{Logarithmic_stability_estimate}
        For any $\mathcal{P}_{1},\mathcal{P}_{2}\in\mathcal{\mathcal{A}}$ and $k_{1},k_{2}\in\mathcal{K}$, there exist some constant $\alpha$ with $0<\alpha<\frac{1}{4}$, some constant $C>0$ and some $0<\rho_{*}<1$, where $C$, $\alpha$ and $\rho_*$ only depend on the prior information $PI$, such that, if $0<\rho:=\left\Vert \Lambda_{k_{1}}^{\mathcal{P}_{1}}-\Lambda_{k_{2}}^{\mathcal{P}_{2}}\right\Vert _{*}<\rho_{*}$, then the following logarithmic stability estimate holds:
        
        \begin{equation}\label{L2_stability_result}
            \left\Vert \gamma_{k_{1}}^{\mathcal{P}_{1}}-\gamma_{k_{2}}^{\mathcal{P}_{2}}\right\Vert _{L^{2}\left(\Omega\right)}
            \leq 
            C\left|\log\left(\left\Vert \Lambda_{k_{1}}^{\mathcal{P}_{1}}-\Lambda_{k_{2}}^{\mathcal{P}_{2}}\right\Vert _{*}\right)\right|^{-\alpha}.      
        \end{equation}
     Here, $\Lambda_{k_{i}}^{\mathcal{P}_{i}},i=1,2$ is the corresponding Dirichlet-to-Neumann map defined in Definition \ref{def2}.
    \end{lemma}

    Similar to the proof of Proposition 3.4 in \cite{BerFra20}, we now use the logarithmic stability to show that not only two polygons have the same number of vertices, but also the distances between vertices are small enough provided two Dirichlet-to-Neumann maps are close enough. We first have the following lemma:
    
    \begin{lemma}\label{maindistancelemma}   
        Given two polygons $\mathcal{P}_{1},\mathcal{P}_{2}\in\mathcal{A}$ and two conductivity constants $k_{1},k_{2}\in\mathcal{K}$, let $\rho=\left\Vert \Lambda_{k_{1}}^{\mathcal{P}_{1}}-\Lambda_{k_{2}}^{\mathcal{P}_{2}}\right\Vert _{*}$ be the difference of the corresponding Dirichlet-to-Neumann map. Let $\rho_*$ be the constant of Lemma \ref{Logarithmic_stability_estimate}. There exist constants $\alpha$ and $C$, depending only on the prior information $PI$, such that, if $0<\rho<\rho_*$, then
        
        \begin{equation}  
            \sqrt{\left|\mathcal{P}_{1}\triangle\mathcal{P}_{2}\right|}\leq C\left|\log\left(\rho\right)\right|^{-\alpha}.
        \end{equation}

        Here, $\left|\mathcal{P}_{1}\triangle\mathcal{P}_{2}\right|$ represents the Lebesgue measure of the symmetric difference between $\mathcal{P}_{1}$ and   $\mathcal{P}_{2}$.  
    \end{lemma}

    \begin{proof}
        If $\rho=0$, uniqueness for the two-dimensional conductivity problem (see \cite{astala2006calderon}) gives $\gamma_{k_{1}}^{\mathcal{P}_{1}}=\gamma_{k_{2}}^{\mathcal{P}_{2}}$ a.e., hence $k_{1}=k_{2}$ and $\mathcal{P}_{1}=\mathcal{P}_{2}$ up to a set of measure zero, and the claim is trivial. We therefore assume $0<\rho<\rho_*$, so that \eqref{L2_stability_result} applies and $\left|\log\rho\right|$ is well defined and positive. We calculate that
        \[
            \gamma_{k_{1}}^{\mathcal{P}_{1}}-\gamma_{k_{2}}^{\mathcal{P}_{2}}=
            \begin{cases}
                0 & \Omega\setminus\left(\mathcal{P}_{1}\cup\mathcal{P}_{2}\right)\\
                k_{1}-1 & \mathcal{P}_{1}\setminus\mathcal{P}_{2}\\
                1-k_{2} & \mathcal{P}_{2}\setminus\mathcal{P}_{1}\\
                k_{1}-k_{2} & \mathcal{P}_{1}\cap\mathcal{P}_{2}
            \end{cases}.
        \]

        Squaring the norm on the left hand side of \eqref{L2_stability_result} and inserting the case decomposition above, we get
        \begin{equation}\label{estimate_P1triangleP2}
            \begin{aligned}
                \left\Vert \gamma_{k_{1}}^{\mathcal{P}_{1}}-\gamma_{k_{2}}^{\mathcal{P}_{2}}\right\Vert _{L^{2}\left(\Omega\right)}^{2}  &=\left(k_{1}-1\right)^{2}\left|\mathcal{P}_{1}\setminus\mathcal{P}_{2}\right|\\
                & +\left(k_{2}-1\right)^{2}\left|\mathcal{P}_{2}\setminus\mathcal{P}_{1}\right|\\
                & +\left(k_{1}-k_{2}\right)^{2}\left|\mathcal{P}_{1}\cap\mathcal{P}_{2}\right|\\
                & \geq\frac{1}{\lambda_{1}^{2}}\left|\mathcal{P}_{1}\triangle\mathcal{P}_{2}\right|.
            \end{aligned}
        \end{equation}

        Combined \eqref{L2_stability_result} and \eqref{estimate_P1triangleP2}, we have 
        \begin{align*}
            \left|\mathcal{P}_{1}\triangle\mathcal{P}_{2}\right| 
            & \leq\lambda_{1}^{2}\left\Vert \gamma_{k_{1}}^{\mathcal{P}_{1}}-\gamma_{k_{2}}^{\mathcal{P}_{2}}\right\Vert _{L^{2}\left(\Omega\right)}^{2}\\
            & \leq\lambda_{1}^{2}C^{2}\left|\log\left(\rho\right)\right|^{-2\alpha},
        \end{align*}

        with $\rho=\left\Vert \Lambda_{k_{1}}^{\mathcal{P}_{1}}-\Lambda_{k_{2}}^{\mathcal{P}_{2}}\right\Vert _{*}$, which complete the proof.
    \end{proof}

    We record a uniform lower bound on the area of admissible polygons, which will be used in Section 6 to control the conductivity difference in the small-data regime.

    \begin{lemma}\label{area lower bound}
        There exists a constant $a_{*}>0$, depending only on the prior information $PI$, such that
        \[
            \left|\mathcal P\right|\geq a_{*}
        \]
        for every $\mathcal P\in\mathcal A$.
    \end{lemma}
    \begin{proof}
        Fix $\mathcal P\in\mathcal A$ and a point $x_0\in\partial\mathcal P$. In the chart of Definition \ref{def1} centered at $x_0$ one has, after a translation so that $x_0$ has coordinates $\left(0,0\right)$,
        \[
            \left(\operatorname{int}\mathcal P\right)\cap Q=\left\{\left(x_1,x_2\right)\in Q:\ x_2>\varphi\left(x_1\right)\right\},\qquad Q:=\left[-r_0,r_0\right]\times\left[-K_0r_0,K_0r_0\right],
        \]
        with $\operatorname{Lip}\varphi\le K_0$ and $\varphi\left(0\right)=0$, hence $\varphi\left(x_1\right)\le K_0\left\lvert x_1\right\rvert$. For every $x_1\in\left[-r_0,r_0\right]$ the vertical segment $\left\{\left(x_1,x_2\right):\ K_0\left\lvert x_1\right\rvert<x_2\le K_0r_0\right\}$ lies in $\left(\operatorname{int}\mathcal P\right)\cap Q$, so
        \[
            \left|\mathcal P\right|\ge\left|\left(\operatorname{int}\mathcal P\right)\cap Q\right|\ge\int_{-r_0}^{r_0}\left(K_0r_0-K_0\left\lvert x_1\right\rvert\right)\mathrm dx_1=K_0r_0^{2},
        \]
        which proves the claim with $a_{*}:=K_0r_0^{2}$.
    \end{proof}

    We also record the uniform separation of non-adjacent sides, which is the geometric fact behind the compactness of the admissible polygon class and the uniform extension constructions used below.

    \begin{lemma}\label{edge separation}
        There exists a constant $d_{*}>0$, depending only on the prior information $PI$, such that for every $\mathcal P\in\mathcal A$ and every pair of non-adjacent sides $e_i,e_j\subset\partial\mathcal P$,
        \[
            \mathrm{dist}\left(e_i,e_j\right)\geq d_{*} .
        \]
    \end{lemma}
    \begin{proof}
        Set $d_{*}:=\min\left\{\frac{r_0}{2},\ \frac{K_0r_0}{2},\ \frac{d_0}{2\sqrt{1+K_0^{2}}}\right\}$. Suppose, for contradiction, that there exist $\mathcal P_n\in\mathcal A$ and non-adjacent sides $e_i^n,e_j^n\subset\partial\mathcal P_n$ with $\mathrm{dist}\left(e_i^n,e_j^n\right)<d_{*}$. Choose $p_n\in e_i^n$ and $q_n\in e_j^n$ with $\left\lvert p_n-q_n\right\rvert=\mathrm{dist}\left(e_i^n,e_j^n\right)<d_{*}$. In the coordinate chart of Definition \ref{def1} centered at $p_n$ (so that $p_n$ has coordinates $\left(0,0\right)$ and $\partial\mathcal P_n$ is the graph of a $K_0$-Lipschitz function $\varphi$ with $\varphi\left(0\right)=0$), both coordinates of $q_n$ have absolute value at most $\left\lvert p_n-q_n\right\rvert$. Since $\left\lvert p_n-q_n\right\rvert<d_{*}\le\min\left\{\frac{r_0}{2},\frac{K_0r_0}{2}\right\}$, the point $q_n$ lies in the chart rectangle, and as $q_n\in\partial\mathcal P_n$ it lies on the graph of $\varphi$. The subarc of this graph between $p_n$ and $q_n$ is a boundary arc of $\mathcal P_n$ whose length is at most
        \[
            \sqrt{1+K_0^{2}}\,\left\lvert p_n-q_n\right\rvert<\sqrt{1+K_0^{2}}\,d_{*}\le\frac{d_0}{2}<d_0 .
        \]
        Because $p_n$ and $q_n$ lie on non-adjacent sides, this boundary arc contains at least one complete side of $\mathcal P_n$: it must travel from the side of $p_n$ to the side of $q_n$ through at least one intervening side. Its length is therefore at least $d_0$, contradicting the bound above. This proves the claim.
    \end{proof}

    We then have the following results (Lemma 3.2 in \cite{BerFra20} and Proposition 3.3 in \cite{BerFra20}):
    
    \begin{lemma}\label{first_distance_estimate}
        For any $\mathcal{P}_{1},\mathcal{P}_{2}\in\mathcal{A}$ and $k_{1},k_{2}\in\mathcal{K}$, there exists some constant $C>0$ only depends on the prior information $PI$ such that there holds
        \begin{equation}
            \text{d}_{H}\left(\partial\mathcal{P}_{1},\partial\mathcal{P}_{2}\right)\leq C\sqrt{\left|\mathcal{P}_{1}\triangle\mathcal{P}_{2}\right|}
        \end{equation}

        where $\text{d}_{H}\left(\partial\mathcal{P}_{1},\partial\mathcal{P}_{2}\right)$ represent the Hausdorff distance between $\partial\mathcal{P}_{1}$ and $\partial\mathcal{P}_{2}$.
    \end{lemma}

    \begin{lemma}\label{same_number_result}
        There exists some positive constants $C$ and $\delta_{\mathrm{geom}}$ only depend on the prior information $PI$, such that if $\text{d}_{H}\left(\partial\mathcal{P}_{1},\partial\mathcal{P}_{2}\right)\leq\delta_{\mathrm{geom}}$, then 
        \begin{enumerate}
            \item 
                $\mathcal{P}_{1}$ and $\mathcal{P}_{2}$ have the same number of
                vertices, in other words, there exists some integer $N$ such that
                $^{\#}\mathcal{P}_{1}={}^{\#}\mathcal{P}_{2}=N$
            \item 
                d$\left(x_{1}^{i},x_{2}^{i}\right)\leq C\text{d}_{H}\left(\partial\mathcal{P}_{1},\partial\mathcal{P}_{2}\right)$,
                here $x_{j}^{i}$ represent the $i$-th vertex of $\mathcal{P}_{j}$, for $j=1,2$ and $i=1,\cdots,N$.
        \end{enumerate}
    \end{lemma}

    Finally, Combined Lemma \ref{maindistancelemma}, Lemma \ref{first_distance_estimate}
    and Lemma \ref{same_number_result}, we have:
    
    \begin{theorem}\label{final_same_number_result}
        Given $\mathcal{P}_{1},\mathcal{P}_{2}\in\mathcal{A}$ and $k_{1},k_{2}\in\mathcal{K}$, let $\rho=\left\Vert \Lambda_{k_{1}}^{\mathcal{P}_{1}}-\Lambda_{k_{2}}^{\mathcal{P}_{2}}\right\Vert _{*}$ be the difference of the corresponding Dirichlet-Neumann map, there exists some constant $0<\delta<1$ only depends on the prior information $PI$, such that if $0<\rho<\delta<1$, then $\mathcal{P}_{1}$ and $\mathcal{P}_{2}$ have the same number of vertices.
    \end{theorem}

\section{Calculation of the derivative}
    In Theorem \ref{final_same_number_result}, we have shown that if the difference between two Dirichlet-to-Neumann maps is small enough, then the polygons have the same number $N\le N_{0}$ of vertices, with corresponding vertices $x_{1}^{i},x_{2}^{i}$. A deformation between $\mathcal{P}_{1}$ and $\mathcal{P}_{2}$ is then defined by the affine map on each vertex, and this local deformation extends to the whole domain $\Omega$. We have the following geometry extension lemma (See the proof of Theorem 4.1 in \cite{hanke2024shape})
    
    \begin{lemma}\label{extension_of_affine_map}
        Given two polygons $\mathcal{P}_{1},\mathcal{P}_{2}\in\mathcal{A}$ with the same number of vertices. Let $V$:$\partial\mathcal{P}_{1}\to\mathbb R^{2}$ be the continuous piecewise-affine vector field defined by $V\left(x_{1}^{i}\right)=x_{2}^{i}-x_{1}^{i}$, where $x_{1}^{i}$ and $x_{2}^{i}$ represent the $i$-th vertices of $\mathcal{P}_{1}$ and $\mathcal{P}_{2}$ respectively (so that $\mathrm{Id}+V$ maps $\partial\mathcal P_1$ onto $\partial\mathcal P_2$). Then $V$ can be extended to $h\in W_{c}^{1,\infty}\left(\Omega\right)$ such that $h|_{\partial\mathcal{P}_1}=V$ and support$\left(h\right)\subset\Omega$. Moreover, there exists a constant $C>0$ depending only on the prior information $PI$ such that
        \begin{equation}
            \left\Vert h\right\Vert _{W^{1,\infty}\left(\Omega\right)}\leq C\max_{i}\left|x_{1}^{i}-x_{2}^{i}\right|\leq C\sum_{i}\left|x_{1}^{i}-x_{2}^{i}\right|,
        \end{equation}
        and, for every $t\in\left[0,1\right]$ with $t\left\Vert h'\right\Vert _{L^{\infty}\left(\Omega\right)}\leq\frac12$, the map $\Phi_{ht}\left(x\right)=x+th\left(x\right)$ is a bi-Lipschitz homeomorphism of $\Omega$ onto itself with $\Phi_{ht}=\text{id}$ near $\partial\Omega$.
    \end{lemma}
    \begin{proof}
        The construction works for every admissible polygon, with constants depending only on $PI$. Write $V_i:=x_2^i-x_1^i$, and let $\lambda_i:\partial\mathcal P_1\to\left[0,1\right]$ be the continuous piecewise-affine hat function with $\lambda_i\left(x_1^j\right)=\delta_{ij}$: $\lambda_i$ is affine on the two sides adjacent to $x_1^i$, decreasing from $1$ at $x_1^i$ to $0$ at the adjacent vertices, and vanishes on the remaining sides. Then $V=\sum_{i=1}^{N}V_i\lambda_i$ on $\partial\mathcal P_1$, and along the boundary $\left\lvert\lambda_i\right\rvert\le1$ with $\left\lvert\frac{d}{ds}\lambda_i\right\rvert\le\frac1{d_0}$ (the sides have length at least $d_0$). Each $\lambda_i$ is Lipschitz on $\partial\mathcal P_1$ with a constant controlled by $PI$ only: if $p,q\in\partial\mathcal P_1$ with $\left\lvert p-q\right\rvert<r_c:=\frac14\min\left\{r_0,K_0r_0\right\}$, then $p$ and $q$ lie in one common graph chart, the boundary arc between them has length at most $\sqrt{1+K_0^2}\left\lvert p-q\right\rvert$, and $\left\lvert\lambda_i\left(p\right)-\lambda_i\left(q\right)\right\rvert\le\frac1{d_0}\sqrt{1+K_0^2}\left\lvert p-q\right\rvert$; if $\left\lvert p-q\right\rvert\ge r_c$, then $\left\lvert\lambda_i\left(p\right)-\lambda_i\left(q\right)\right\rvert\le1\le r_c^{-1}\left\lvert p-q\right\rvert$. Set $\eta:=\frac{d_0}{4}$ and $U_\eta:=\left\{x\in\mathbb R^2:\ \mathrm{dist}\left(x,\partial\mathcal P_1\right)<\eta\right\}$, define $f_i$ on $A:=\partial\mathcal P_1\cup\left(\mathbb R^2\setminus U_\eta\right)$ by $f_i=\lambda_i$ on $\partial\mathcal P_1$ and $f_i=0$ on $\mathbb R^2\setminus U_\eta$, and let
        \[
            \widehat\lambda_i\left(x\right):=\min\left\{1,\ \max\left\{0,\ \inf_{z\in A}\left(f_i\left(z\right)+L_0\left\lvert x-z\right\rvert\right)\right\}\right\},
            \qquad
            L_0:=\max\left\{\frac{\sqrt{1+K_0^2}}{d_0},\ \frac1{r_c},\ \frac4{d_0}\right\}.
        \]
        The two parts of $A$ stay at distance $\eta$, so $f_i$ is $L_0$-Lipschitz on $A$, and the McShane extension above is $L_0$-Lipschitz on $\mathbb R^2$ with $\widehat\lambda_i=f_i$ on $A$; truncation preserves this, so $\widehat\lambda_i=\lambda_i$ on $\partial\mathcal P_1$, $\widehat\lambda_i=0$ on $\mathbb R^2\setminus U_\eta$, and $\left\lVert\widehat\lambda_i\right\rVert_{W^{1,\infty}\left(\mathbb R^2\right)}\le\max\left\{1,L_0\right\}$. Since $\mathrm{dist}\left(\partial\mathcal P_1,\partial\Omega\right)\ge d_0>\eta$, the set $\overline{U_\eta}$ is compactly contained in $\Omega$. Define
        \[
            h:=\sum_{i=1}^{N}V_i\,\widehat\lambda_i .
        \]
        Then $h\in W_c^{1,\infty}\left(\Omega;\mathbb R^2\right)$, $h|_{\partial\mathcal P_1}=\sum_iV_i\lambda_i=V$, and $\left\lVert h\right\rVert_{W^{1,\infty}}\le\sum_i\left\lvert V_i\right\rvert\left\lVert\widehat\lambda_i\right\rVert_{W^{1,\infty}}\le C\sum_i\left\lvert x_1^i-x_2^i\right\rvert$ with $C=C\left(PI\right)$. Finally, if $t\left\lVert h'\right\rVert_{L^\infty}\le\frac12$, then $\Phi_{ht}=\mathrm{Id}+th$ is injective, since $\left\lvert\Phi_{ht}\left(x\right)-\Phi_{ht}\left(y\right)\right\rvert\ge\frac12\left\lvert x-y\right\rvert$; being equal to the identity on $\partial\Omega$ and proper, it maps $\Omega$ onto itself, hence it is a bi-Lipschitz homeomorphism of $\Omega$. The change-of-variables formula therefore applies to $\Phi_{ht}$.
    \end{proof}
    
 \begin{figure}[h]
    \centering
    \includegraphics[width=0.8\textwidth]{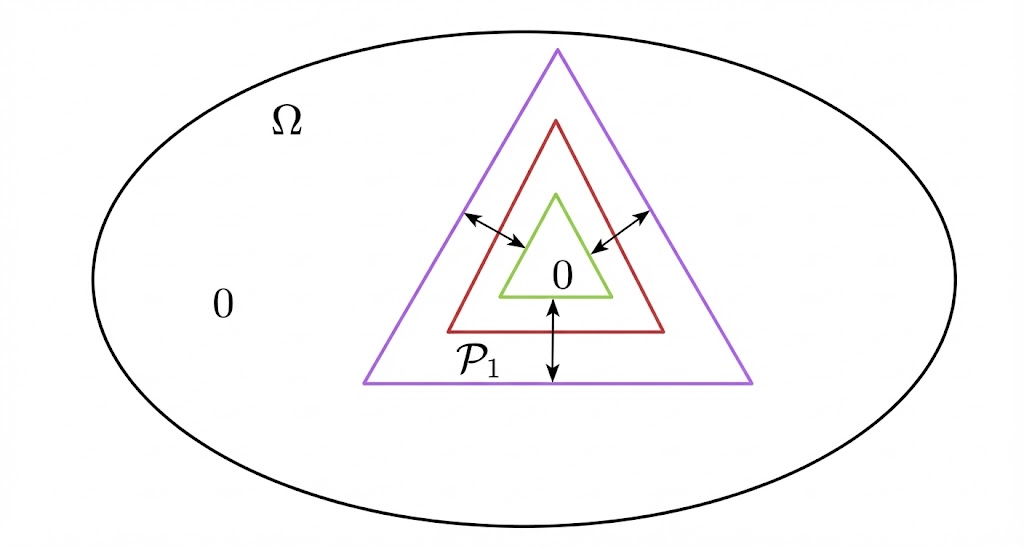}
    \caption{The vertex field $V$ on $\partial\mathcal{P}_{1}$ is extended inside and outside the inclusion by linear combination of the boundary hat functions, and set to zero beyond the band between the two ``shape-like'' polygons; the band and its boundary curves are schematic.}
    \label{fig:explainimg}
 \end{figure}
      
    The map $h$ of Lemma \ref{extension_of_affine_map} defines a deformation path between the two configurations. Throughout the rest of this section we assume
    \[
        \left\lVert Dh\right\rVert_{L^{\infty}\left(\Omega\right)}\le\frac12 ,
    \]
    which is satisfied in all applications (see Section 6). Under this assumption, for every $t\in\left[0,1\right]$ the map $\Phi_{ht}:\Omega\to\Omega$, $x\mapsto x+th\left(x\right)$, is a bi-Lipschitz homeomorphism of $\Omega$ onto itself that is the identity near $\partial\Omega$ (Lemma \ref{extension_of_affine_map}); maps of this form are Lipschitz with Lipschitz inverse, hence differentiable almost everywhere, and the change-of-variables formula applies to them. Since $h|_{\partial\mathcal P_1}$ is affine on each side and takes the value $x_{2}^{i}-x_{1}^{i}$ at the vertex $x_1^i$, the map $\Phi_{h1}$ sends each side of $\partial\mathcal P_1$ onto the corresponding side of $\partial\mathcal P_2$; being the identity near $\partial\Omega$, it sends the region bounded by $\partial\mathcal P_1$, namely $\mathcal P_1$, onto the region bounded by $\partial\mathcal P_2$, namely $\mathcal P_2$. We write
    \[
        \mathcal P(t):=\Phi_{ht}\left(\mathcal P_1\right),\qquad \kappa:=k_2-k_1,\qquad k(t):=k_1+t\kappa ,
    \]
    so that $\mathcal P(0)=\mathcal P_1$, $\mathcal P(1)=\mathcal P_2$, $k(0)=k_1$ and $k(1)=k_2$. Since $k_1,k_2\in\left(\frac1{\lambda_0},\lambda_0\right)$, convexity gives $\frac1{\lambda_0}<k(t)<\lambda_0$ for every $t\in\left[0,1\right]$, and because $\Phi_{ht}$ is the identity near $\partial\Omega$, $\gamma_{k(t)}^{\mathcal P(t)}=1$ in a neighborhood of $\partial\Omega$. The conductivities $\gamma_{k(t)}^{\mathcal P(t)}$ are therefore uniformly elliptic and uniformly bounded, and all estimates below apply uniformly in $t\in\left[0,1\right]$.

    At the base point $\left(\mathcal P_1,k_1\right)$ the relevant Dirichlet problems are, for any $\phi,\psi\in H^{1/2}\left(\partial\Omega\right)$,
    \begin{equation}\label{dirichlet_boundary_problem_phi}
        \begin{cases}
            \mathrm{div}\left(\gamma_{k_{1}}^{\mathcal P_1}\nabla u_0^\varphi\right)=0 & \text{in }\Omega\\
            u_0^\varphi=\phi & \text{on }\partial\Omega,
        \end{cases}
    \end{equation}
    and
    \begin{equation}\label{Dirichlet_boundary_problem_psi}
        \begin{cases}
            \mathrm{div}\left(\gamma_{k_{1}}^{\mathcal P_1}\nabla u_0^\psi\right)=0 & \text{in }\Omega\\
            u_0^\psi=\psi & \text{on }\partial\Omega .
        \end{cases}
    \end{equation}

    For any $t\in\left[0,1\right]$ and $\forall \phi,\psi\in H^{1/2}\left(\partial\Omega\right)$, let
    \begin{align}\label{functionalform}
        F\left(t,\phi,\psi\right) & =\left\langle \Lambda_{k(t)}^{\mathcal{P}(t)}\left(\phi\right),\psi\right\rangle _{H^{-1/2}\left(\partial\Omega\right)\times H^{1/2}\left(\partial\Omega\right)}
    \end{align}

    We derive the first variation of $F$ at $t=0$ using the pull-back formulation on the reference configuration. We have the following main theorem:

    \begin{theorem}\label{mainderivativestheorem}
        Let $\left(\mathcal P_1,k_1\right)\in\mathcal A\times\mathcal K$ be an admissible base point, let $\kappa\in\mathbb R$ satisfy $\frac1{\lambda_0}\le k_1+t\kappa\le\lambda_0$ for every $t\in\left[0,1\right]$, and let $h\in W_c^{1,\infty}\left(\Omega;\mathbb R^2\right)$ be any field with $\left\lVert Dh\right\rVert_{L^{\infty}\left(\Omega\right)}\le\frac12$. Write $\mathcal P(t):=\left(\mathrm{Id}+th\right)\left(\mathcal P_1\right)$, $k(t):=k_1+t\kappa$ and $\gamma_{k(t)}^{\mathcal P(t)}:=1+\left(k(t)-1\right)\chi_{\mathcal P(t)}$. For every $t\in\left[0,1\right]$, the map $\Phi_{ht}$ is bi-Lipschitz and equals the identity near $\partial\Omega$. Together with $\lambda_0^{-1}\le k(t)\le\lambda_0$, this ensures that $\gamma_{k(t)}^{\mathcal P(t)}$ is uniformly elliptic and equals $1$ near $\partial\Omega$. The corresponding Dirichlet-to-Neumann map is therefore well defined. For $\phi,\psi\in H^{1/2}\left(\partial\Omega\right)$ and $t\in\left[0,1\right]$ set $F\left(t,\phi,\psi\right):=\left\langle \Lambda_{k(t)}^{\mathcal{P}(t)}\phi,\psi\right\rangle_{H^{-1/2}\left(\partial\Omega\right)\times H^{1/2}\left(\partial\Omega\right)}$. Then $t\mapsto F\left(t,\phi,\psi\right)$ is differentiable on $\left[0,1\right]$, with
        \begin{equation}\label{derivatives_at_zero}
            \frac{dF}{dt}\Big|_{t=0}
            =\left(k_{1}-1\right)\int_{\partial\mathcal{P}_{1}}\left(h\cdot\nu\right)\nabla u_0^\varphi\cdot\left(M\nabla u_0^\psi\right)\mathrm ds
            +\kappa\int_{\mathcal{P}_{1}}\nabla u_0^\varphi\cdot\nabla u_0^\psi\,\mathrm dx,
        \end{equation}
        where $u_0^\varphi$ and $u_0^\psi$ are the base point solutions of \eqref{dirichlet_boundary_problem_phi} and \eqref{Dirichlet_boundary_problem_psi}, and the gradient traces in the first integral are taken on the exterior side $\Omega\setminus\overline{\mathcal P_1}$ of $\partial\mathcal{P}_{1}$.
        The vector fields $\tau$ and $\nu$ are the tangential and outer normal unit fields on $\partial\mathcal{P}_{1}$, and $M$ is the symmetric $2\times2$ matrix defined on $\partial\mathcal{P}_{1}$ by $M:=\tau\otimes\tau+\frac1{k_1}\nu\otimes\nu$, with eigenvalues $1$ and $\frac{1}{k_{1}}$ in the directions $\tau$ and $\nu$.
    
        Moreover, there exists some constant $C>0$ which only depends on the prior information $PI$ such that, for every $t\in\left[0,1\right]$,
        \begin{equation}\label{continuity_of_derivatives}
            \left|\frac{dF}{dt}\left(t,\phi,\psi\right)-\frac{dF}{dt}\left(0,\phi,\psi\right)\right|\leq C\left(\left\lVert h\right\rVert _{W^{1,\infty}\left(\Omega\right)}+\left\lvert\kappa\right\rvert\right)^{2}t\,\left\lVert \phi\right\rVert _{H^{1/2}\left(\partial\Omega\right)}\left\lVert \psi\right\rVert _{H^{1/2}\left(\partial\Omega\right)}
        \end{equation}
    \end{theorem}

    \begin{proof}
        We work in the reference configuration and use the pull-back (transported) formulation. Set $\gamma_0:=\gamma_{k_1}^{\mathcal P_1}$ and
        \[
            a_t:=1+\left(k_1+t\kappa-1\right)\chi_{\mathcal P_1}=\gamma_0+t\kappa\,\chi_{\mathcal P_1},\qquad A_t:=D\Phi_{ht}^{-1}D\Phi_{ht}^{-T}\det D\Phi_{ht},\qquad C_t:=a_tA_t .
        \]
        Let $u_t^f$ be the solution of $\mathrm{div}\left(\gamma_{k(t)}^{\mathcal P(t)}\nabla u_t^f\right)=0$ in $\Omega$ with $u_t^f=f$ on $\partial\Omega$, and set $U_t^f:=u_t^f\circ\Phi_{ht}$. Since $\gamma_{k(t)}^{\mathcal P(t)}\left(\Phi_{ht}\left(y\right)\right)=a_t\left(y\right)$, the change of variables $x=\Phi_{ht}\left(y\right)$ gives the transported equation
        \begin{equation}\label{eq:transported weak form}
            \int_\Omega C_t\,\nabla U_t^f\cdot\nabla w\,\mathrm dx=0\qquad\forall w\in H_0^1\left(\Omega\right),
        \end{equation}
        and $U_t^f=f$ on $\partial\Omega$ because $\Phi_{ht}=\mathrm{id}$ near $\partial\Omega$. Moreover,
        \begin{equation}\label{eq:pullback functional}
            F\left(t,\phi,\psi\right)=\int_\Omega C_t\,\nabla U_t^\phi\cdot\nabla U_t^\psi\,\mathrm dx :
        \end{equation}
        indeed, writing $A_t\nabla U_t=D\Phi_{ht}^{-1}\det D\Phi_{ht}\,\nabla u_t\left(\Phi_{ht}\right)$ and $\nabla U_t=D\Phi_{ht}^T\nabla u_t\left(\Phi_{ht}\right)$, the right hand side equals $\int_\Omega\gamma_{k(t)}^{\mathcal P(t)}\nabla u_t^\phi\cdot\nabla u_t^\psi\,\mathrm dx$, which is $\left\langle\Lambda_{k(t)}^{\mathcal P(t)}\phi,\psi\right\rangle$ by Green's identity, since both conductivities equal $1$ near $\partial\Omega$.

        We next record the differentiability of the transported solution map, which justifies differentiating \eqref{eq:pullback functional}.

        \begin{lemma}\label{solution differentiability}
            For every $f\in H^{1/2}\left(\partial\Omega\right)$, the map $t\mapsto U_t^f$ is differentiable on $\left(0,1\right)$ into $H^1\left(\Omega\right)$, with one-sided derivatives at $t=0$ and $t=1$. Its derivative $\dot U_t^f\in H_0^1\left(\Omega\right)$ is the unique solution of
            \begin{equation}\label{eq:derivative equation}
                \int_\Omega C_t\,\nabla\dot U_t^f\cdot\nabla w\,\mathrm dx=-\int_\Omega C_t'\,\nabla U_t^f\cdot\nabla w\,\mathrm dx\qquad\forall w\in H_0^1\left(\Omega\right) .
            \end{equation}
        \end{lemma}
        \begin{proof}
            For $s\ne0$ set $Q_s:=\frac{U_{t+s}^f-U_t^f}{s}\in H_0^1\left(\Omega\right)$ (the boundary data are independent of $t$). Subtracting \eqref{eq:transported weak form} at $t+s$ and at $t$ gives
            \[
                \int_\Omega C_t\,\nabla Q_s\cdot\nabla w\,\mathrm dx=-\int_\Omega\frac{C_{t+s}-C_t}{s}\,\nabla U_{t+s}^f\cdot\nabla w\,\mathrm dx\qquad\forall w\in H_0^1\left(\Omega\right) .
            \]
            Since $C_t$ is uniformly coercive, $\left\lVert\frac{C_{t+s}-C_t}{s}\right\rVert_{L^\infty}\le C\left(\left\lVert h\right\rVert_{W^{1,\infty}}+\left\lvert\kappa\right\rvert\right)$ and $\left\lVert\nabla U_{t+s}^f\right\rVert_{L^2}\le C\left\lVert f\right\rVert_{H^{1/2}}$, testing with $w=Q_s$ gives a bound on $\left\lVert Q_s\right\rVert_{H_0^1}$ that is uniform in $s$. Moreover $\frac{C_{t+s}-C_t}{s}\to C_t'$ in $L^\infty\left(\Omega\right)$, and, subtracting \eqref{eq:transported weak form} at $t+s$ and at $t$ and testing with $U_{t+s}^f-U_t^f$, the uniform coercivity of $C_t$ and $\left\lVert C_{t+s}-C_t\right\rVert_{L^\infty}\le C\left\lvert s\right\rvert\left(\left\lVert h\right\rVert_{W^{1,\infty}}+\left\lvert\kappa\right\rvert\right)$ give
            \[
                \left\lVert U_{t+s}^f-U_t^f\right\rVert_{H^1}\le C\left\lvert s\right\rvert\left\lVert f\right\rVert_{H^{1/2}},
            \]
            so $U_{t+s}^f\to U_t^f$ in $H^1\left(\Omega\right)$ as $s\to0$. Consequently every accumulation point $\dot U$ of $Q_s$ in the weak topology of $H_0^1\left(\Omega\right)$ satisfies \eqref{eq:derivative equation}, and by the coercivity of $C_t$ the solution of \eqref{eq:derivative equation} is unique. Subtracting \eqref{eq:derivative equation} from the equation of $Q_s$ and testing with $w=e_s:=Q_s-\dot U_t^f$, we obtain, writing $\lambda>0$ for the uniform coercivity constant of $C_t$,
            \begin{equation}
                \begin{aligned}
                    \lambda\left\lVert\nabla e_s\right\rVert_{L^2}^{2}
                    \le\int_\Omega C_t\,\nabla e_s\cdot\nabla e_s\,\mathrm dx
                    \le\Bigg[
                    &\left\lVert\frac{C_{t+s}-C_t}{s}-C_t'\right\rVert_{L^\infty}\left\lVert\nabla U_{t+s}^f\right\rVert_{L^2}\\
                    &+\left\lVert C_t'\right\rVert_{L^\infty}\left\lVert\nabla\left(U_{t+s}^f-U_t^f\right)\right\rVert_{L^2}
                    \Bigg]\left\lVert\nabla e_s\right\rVert_{L^2},
                \end{aligned}
            \end{equation}
            where the bracketed factor is $o\left(1\right)$ as $s\to0$: $\frac{C_{t+s}-C_t}{s}\to C_t'$ in $L^\infty$, $\left\lVert\nabla U_{t+s}^f\right\rVert_{L^2}\le C\left\lVert f\right\rVert_{H^{1/2}}$ is bounded, and $\left\lVert\nabla\left(U_{t+s}^f-U_t^f\right)\right\rVert_{L^2}\to0$ by the estimate above. Dividing by $\left\lVert\nabla e_s\right\rVert_{L^2}$ (if nonzero; otherwise there is nothing to prove) gives $e_s\to0$ strongly in $H^1\left(\Omega\right)$. Hence $Q_s\to\dot U_t^f$ strongly in $H^1\left(\Omega\right)$, which proves the differentiability of $t\mapsto U_t^f$ on $\left(0,1\right)$; the same argument at $t=0$ and $t=1$ with $s\to0^{\pm}$ gives the one-sided derivatives.
        \end{proof}

        Since the boundary data are independent of $t$, \eqref{eq:derivative equation} gives $\dot U_t^f\in H_0^1\left(\Omega\right)$ for every $t$. Differentiating \eqref{eq:pullback functional} and using \eqref{eq:transported weak form} with the test functions $\dot U_t^\psi$ and $\dot U_t^\phi$ (the matrix $C_t$ is symmetric), the two terms containing $\dot U_t$ vanish, and we obtain the Hadamard structure
        \begin{equation}\label{eq:Hadamard structure}
            \frac{dF}{dt}\left(t\right)=\int_\Omega C_t'\,\nabla U_t^\phi\cdot\nabla U_t^\psi\,\mathrm dx ,
        \end{equation}
        where $C_t'=\frac{d}{dt}C_t$. At $t=0$ we have
        \[
            A_0=I,\qquad A_0'=\mathcal A_h:=\left(\nabla\cdot h\right)I-h'-\left(h'\right)^T,\qquad C_0'=\kappa\,\chi_{\mathcal P_1}I+\gamma_0\,\mathcal A_h ,
        \]
        where the explicit formula
        \begin{equation}\label{equationofAt}
            A_t=\left(I+th'\right)^{-1}\left(I+t\left(h'\right)^T\right)^{-1}\left(1+t\nabla\cdot h+t^2\det h'\right)
        \end{equation}
        is used. Therefore
        \begin{equation}\label{eq:first variation distributed}
            \frac{dF}{dt}\Big|_{t=0}=\int_\Omega\gamma_0\nabla u_0^\varphi\cdot\left(\mathcal A_h\nabla u_0^\psi\right)\mathrm dx+\kappa\int_{\mathcal P_1}\nabla u_0^\varphi\cdot\nabla u_0^\psi\,\mathrm dx ,
        \end{equation}
        where $u_0^\varphi:=U_0^\varphi$ and $u_0^\psi:=U_0^\psi$ are the base point solutions of \eqref{dirichlet_boundary_problem_phi} and \eqref{Dirichlet_boundary_problem_psi}. Finally, by \cite{BerMicPerSan18} (Proof of Proposition 2.3; see also \cite{BerFraVes17}, Corollary 2.2),
        \begin{equation}\label{eq:relationshipbetweenAhandmatrix}
            \int_\Omega\gamma_0\nabla u_0^\varphi\cdot\left(\mathcal A_h\nabla u_0^\psi\right)\mathrm dx=\left(k_1-1\right)\int_{\partial\mathcal P_1}\left(h\cdot\nu\right)\nabla u_0^\varphi\cdot\left(M\nabla u_0^\psi\right)\mathrm ds ,
        \end{equation}
        with $M$ as in the statement, and \eqref{eq:first variation distributed} combined with \eqref{eq:relationshipbetweenAhandmatrix} yields \eqref{derivatives_at_zero}.

        It remains to prove the continuity estimate \eqref{continuity_of_derivatives}. Standard elliptic estimates applied to the transported problems give
        \begin{equation}\label{eq:stability of transported solutions}
            \left\lVert\nabla\left(U_t^f-U_0^f\right)\right\rVert_{L^2\left(\Omega\right)}\le Ct\left(\left\lVert h\right\rVert_{W^{1,\infty}}+\left\lvert\kappa\right\rvert\right)\left\lVert f\right\rVert_{H^{1/2}\left(\partial\Omega\right)},\qquad \left\lVert\nabla U_t^f\right\rVert_{L^2\left(\Omega\right)}\le C\left\lVert f\right\rVert_{H^{1/2}\left(\partial\Omega\right)},
        \end{equation}
        uniformly in $t\in\left[0,1\right]$: subtracting \eqref{eq:transported weak form} at $t$ and at $0$ and testing with $U_t^f-U_0^f\in H_0^1\left(\Omega\right)$, the leading coefficient $C_0$ is uniformly elliptic while $\left\lVert C_t-C_0\right\rVert_{L^\infty}\le Ct\left(\left\lVert h\right\rVert_{W^{1,\infty}}+\left\lvert\kappa\right\rvert\right)$. Moreover, from $C_t=\left(\gamma_0+t\kappa\,\chi_{\mathcal P_1}\right)A_t$ and \eqref{equationofAt} we get
        \begin{equation}\label{eq:lipschitz of Ct}
            \left\lVert C_t'-C_0'\right\rVert_{L^\infty}\le Ct\left(\left\lVert h\right\rVert_{W^{1,\infty}}^2+\left\lvert\kappa\right\rvert\left\lVert h\right\rVert_{W^{1,\infty}}\right),\qquad \left\lVert C_t'\right\rVert_{L^\infty}\le C\left(\left\lVert h\right\rVert_{W^{1,\infty}}+\left\lvert\kappa\right\rvert\right).
        \end{equation}
        By \eqref{eq:Hadamard structure},
        \[
            \frac{dF}{dt}\left(t\right)-\frac{dF}{dt}\left(0\right)=\int_\Omega\left(C_t'-C_0'\right)\nabla U_t^\phi\cdot\nabla U_t^\psi\,\mathrm dx+\int_\Omega C_0'\nabla\left(U_t^\phi-U_0^\phi\right)\cdot\nabla U_t^\psi\,\mathrm dx+\int_\Omega C_0'\nabla U_0^\phi\cdot\nabla\left(U_t^\psi-U_0^\psi\right)\mathrm dx .
        \]
        Combining this with \eqref{eq:stability of transported solutions} and \eqref{eq:lipschitz of Ct} yields
        \[
            \left|\frac{dF}{dt}\left(t\right)-\frac{dF}{dt}\left(0\right)\right|\le Ct\left(\left\lVert h\right\rVert_{W^{1,\infty}}+\left\lvert\kappa\right\rvert\right)^2\left\lVert\phi\right\rVert_{H^{1/2}\left(\partial\Omega\right)}\left\lVert\psi\right\rVert_{H^{1/2}\left(\partial\Omega\right)},
        \]
        which is \eqref{continuity_of_derivatives}.
    \end{proof}
    
\section{Lower bound of the derivative at $t=0$}
    In Theorem \ref{mainderivativestheorem}, we have shown that for the functional $F$ defined in \eqref{functionalform} for any Dirichlet boundary measurements $\phi$, $\psi$ and any $t\in \lbrack0,1\rbrack$, its derivative at $t=0$ is not only linear, but also continuous associated with the $H^{1/2}\left(\partial\Omega\right)$ norm of corresponding $\phi\text{ and } \psi$. In other words, we have obtained an upper bound for $\frac{dF}{dt}|_{t=0}$. In this section, we bound $\frac{dF}{dt}|_{t=0}$ from below.

    For a polygon $\mathcal P$ with vertices $x_1,\dots,x_N$, a vector $W=\left(W_1,\dots,W_N\right)\in\left(\mathbb R^2\right)^N$, $k\in\mathcal K$ and $\kappa\in\mathbb R$, let $V_W$ be the continuous piecewise-affine vector field on $\partial\mathcal P$ defined by $V_W\left(x_i\right)=W_i$, and set
    \begin{equation}\label{eq:Bdef}
        B_{\mathcal P,k}\left[W,\kappa\right]\left(\phi,\psi\right):=\left(k-1\right)\int_{\partial\mathcal P}\left(V_W\cdot\nu\right)\nabla u^{\phi,e}\cdot\left(M\nabla u^{\psi,e}\right)\mathrm ds+\kappa\int_{\mathcal P}\nabla u^\phi\cdot\nabla u^\psi\,\mathrm dx ,
    \end{equation}
    where $u^\phi$ and $u^\psi$ solve $\mathrm{div}\left(\gamma_{\mathcal P}^{k}\nabla u\right)=0$ in $\Omega$ with Dirichlet data $\phi$ and $\psi$ (the base point problems \eqref{dirichlet_boundary_problem_phi} and \eqref{Dirichlet_boundary_problem_psi}, written there for $\left(\mathcal P_1,k_1\right)$, with $\left(\mathcal P_1,k_1\right)$ replaced by $\left(\mathcal P,k\right)$), $u^{\phi,e},u^{\psi,e}$ denote their exterior traces on $\partial\mathcal P$ (the side of $\Omega\setminus\mathcal P$), and $M$ is the matrix of Theorem \ref{mainderivativestheorem}. By Theorem \ref{mainderivativestheorem}, for $\kappa=k_2-k_1$ and $W_i=x_2^i-x_1^i$ the derivative of $F$ equals
    \[
        \frac{d}{dt}F\left(t,\phi,\psi\right)\Big|_{t=0}=B_{\mathcal P_1,k_1}\left[W,\kappa\right]\left(\phi,\psi\right),
    \]
    because $h\cdot\nu=V_W\cdot\nu$ on $\partial\mathcal P_1$. For a bounded bilinear form $B$ on $H^{1/2}\left(\partial\Omega\right)\times H^{1/2}\left(\partial\Omega\right)$ we write $\left\lVert B\right\rVert_*$ for its operator norm. We have the following result:

    \begin{theorem}\label{lowerbound result}
        There exists a constant $m_{1}>0$ depending only on the prior information $PI$ such that, for every admissible base point $\left(\mathcal P,k\right)\in\mathcal A\times\mathcal K$, every vertex direction $W=\left(W_1,\dots,W_N\right)\in\left(\mathbb R^2\right)^N$ and every $\kappa\in\mathbb R$,
        \begin{equation}\label{lowerbound operator version}
            \left\lVert B_{\mathcal P,k}\left[W,\kappa\right]\right\rVert_*\ge m_{1}\left(\sum_{i=1}^{N}\left\lvert W_i\right\rvert+\left\lvert\kappa\right\rvert\right) .
        \end{equation}
        In particular, for $\mathcal P=\mathcal P_1$, $k=k_1$, $W_i=x_2^i-x_1^i$ and $\kappa=k_2-k_1$, there exist non-zero $\tilde{\phi},\tilde{\psi}\in H^{1/2}\left(\partial\Omega\right)$ such that
        \begin{equation}\label{lowerbound}
            \left| \frac{d}{dt}F\left(t,\tilde{\phi},\tilde{\psi}\right)_{|_{t=0}}\right|\geq \frac{m_{1}}{2}\left(\sum_{i=1}^{N}\left\lvert x_1^i-x_2^i\right\rvert+\left\lvert k_2-k_1\right\rvert\right)\lVert\tilde{\phi}\rVert_{H^{1/2}\left(\partial\Omega\right)}\lVert\tilde{\psi}\rVert_{H^{1/2}\left(\partial\Omega\right)} .
        \end{equation}
    \end{theorem}

    \begin{proof}
        The proof proceeds in three steps: joint injectivity, continuity on the tangent sphere, and compactness.
    
    Fix a point $x_c\in\Omega$ and let $\Omega_{0}:=B_{2L}\left(x_c\right)$ be the open ball of radius $2L$ centered at $x_c$; since $\mathrm{diam}\left(\Omega\right)\le L$, one has $\overline\Omega\subset B_{L}\left(x_c\right)$ and $\mathrm{dist}\left(\overline\Omega,\partial\Omega_{0}\right)\ge L$. Choose a non-empty open ball $U$ with $\overline U\subset\Omega_{0}\setminus\overline\Omega$. We extend the conductivity $\gamma^{k_{1}}_{\mathcal{P}_{1}}$ by setting it equal to $1$ on $\Omega_{0}\setminus\Omega$; to simplify the notation, we keep writing $\gamma$ for the extended coefficient. We construct the Green function on the fixed extension domain $\Omega_0$.
    
    Let $G\left(\cdot,y\right)$ be the Green function in $\Omega_{0}$ associated with the conductivity $\gamma$, in other words, $G$ satisfies the following equation:

    \[
        \begin{cases}
            -\text{div}\left(\gamma\nabla G\left(\cdot,y\right)\right)=\delta_{y} & \text{in }\Omega_{0}\\
            G\left(\cdot,y\right)=0 & \text{on }\partial\Omega_{0}.
        \end{cases}
    \]

    We have some quantitative results for such Green function, see Theorem 4.2 in \cite{alessandrini2005lipschitz} and Proposition 3.4 in \cite{BerFra11}:
    \begin{lemma}\label{global behavior lemma}
        Let $G\left(\cdot,y\right)$ be the Green function of $\mathrm{div}\left(\gamma\nabla\cdot\right)$ in $\Omega_{0}$ with homogeneous Dirichlet condition, where $\gamma$ satisfies $\frac{1}{\lambda_{0}}\leq\gamma\leq\lambda_{0}$. There exists a constant $C>0$, depending only on the prior information $PI$, such that for every $x,y\in\Omega_{0}$, $x\neq y$,
        \begin{equation}\label{global behavior}
            \lvert G\left(x,y\right)\rvert\leq C\left(1+\lvert\log\lvert x-y\rvert\rvert\right),
        \end{equation}
        and, for every $0<r<\frac{1}{2}\mathrm{dist}\left(y,\partial\Omega_{0}\right)$,
        \begin{equation}\label{global behavior energy}
            \lVert G\left(\cdot,y\right)\rVert_{H^{1}\left(\Omega_{0}\setminus B_{r}\left(y\right)\right)}\leq C\left(1+\lvert\log r\rvert\right)^{1/2}.
        \end{equation}
    \end{lemma}

    Near the polygonal vertices, the Green-function gradient may exhibit singularities of order $\varrho^{\mu-1}$, where $\varrho$ is the distance to the vertex and $1/2<\mu<1$. These singularities are square-integrable along the sides and are treated in Step 1 below. Away from the vertices, Lemma \ref{local behavior lemma} gives the local transmission asymptotics.
    \begin{lemma}\label{local behavior lemma}
        Let $G^e:=G|_{\overline{\Omega_0\setminus\mathcal P_1}}$ and $G^i:=G|_{\overline{\mathcal P_1}}$ denote the exterior and interior traces of the Green function. Set
        \[
            \rho_0:=\frac18\min\left\{d_0,d_*,d_0\sin\alpha_0\right\},
            \qquad
            r_{1}:=\frac1{16}\min\left\{\rho_0\sin\alpha_0,d_*,d_0\right\},
        \]
        which depend only on $PI$, and note that $r_1<\rho_0/4\le d_0/32$ and $r_1\le d_0/16$. There exists a constant $C>0$, depending only on $PI$, such that, for every point $y$ in the interior of a side of $\mathcal P_1$ with $\mathrm{dist}\left(y,\text{vertices}\right)\ge\rho_0$ (such points exist: each side has length at least $d_0>2\rho_0$, and after removing the two caps of length $d_0/4$ at the endpoints every remaining point stays at distance at least $\rho_0$ from all vertices — the endpoints at distance at least $d_0/4\ge\rho_0$, every vertex of a non-adjacent side at distance at least $d_*$ by Lemma \ref{edge separation}, and for the triangle $N=3$ the third vertex at distance at least $d_0\sin\alpha_0$ from the line through the opposite side, by the side-length and angle bounds; in addition $B_{r_1}\left(y\right)\subset\Omega$ since $\mathrm{dist}\left(y,\partial\Omega\right)\ge d_0$), and for every $0<r<r_1$, the point $y_r=y+\nu r$ (on the exterior side, $\nu$ being the outer normal of the side) satisfies the following one-sided asymptotics: for every $x$ in the inclusion-side neighborhood $B_{r_1}\left(y\right)\cap\overline{\mathcal P_1}$,
        \begin{equation}\label{local behavior function}
            \left\lvert G^i\left(x,y_r\right)-\frac{2}{k_{1}+1}\Gamma\left(x,y_r\right)\right\rvert\leq C ,
        \end{equation}
        and
        \begin{equation}\label{local behavior gradient}
            \left\lvert \nabla G^i\left(x,y_r\right)-\frac{2}{k_{1}+1}\nabla \Gamma\left(x,y_r\right)\right\rvert\leq C .
        \end{equation}
        Moreover, on the interface segment $B_{r_1}\left(y\right)\cap\partial\mathcal P_1$ the two traces satisfy the transmission conditions
        \begin{equation}\label{transmission relation}
            G^e=G^i,\qquad \partial_{\tau}G^e=\partial_{\tau}G^i,\qquad \partial_{\nu}G^e=k_1\,\partial_{\nu}G^i,\qquad\text{hence}\qquad M\nabla G^e=\nabla G^i ,
        \end{equation}
        where $\tau,\nu$ are the tangential and outer normal directions of the side and $M$ is the matrix of Theorem \ref{mainderivativestheorem}. Here $\Gamma(x,y)$ represents the fundamental solution of the Laplace equation. The pole $y_r$ lies in the exterior phase of conductivity $1$. The corresponding transmitted singularity in the inclusion has coefficient $\frac{2}{1+k_1}$.
    \end{lemma}
    \begin{proof}
        In local coordinates flatten the side to $\left\{x_2=0\right\}$, with the exterior half-plane $\left\{x_2>0\right\}$ carrying conductivity $1$, the interior half-plane $\left\{x_2<0\right\}$ carrying conductivity $k_1$, and the pole at $y_r=\left(0,r\right)$. The method of images produces the exact transmission solution of this half-plane model,
        \[
            G^e_{\mathrm{flat}}\left(x\right)=\Gamma\left(x-y_r\right)+\frac{1-k_1}{1+k_1}\,\Gamma\left(x-y_r^{*}\right)\quad\left(x_2>0\right),\qquad G^i_{\mathrm{flat}}\left(x\right)=\frac{2}{1+k_1}\,\Gamma\left(x-y_r\right)\quad\left(x_2<0\right),
        \]
        where $y_r^{*}=\left(0,-r\right)$ is the reflected pole. Indeed $-\Delta G^e_{\mathrm{flat}}=\delta_{y_r}$ on $\left\{x_2>0\right\}$ and $-k_1\Delta G^i_{\mathrm{flat}}=0$ on $\left\{x_2<0\right\}$; on $\left\{x_2=0\right\}$ the symmetry $\Gamma\left(\cdot,y_r\right)=\Gamma\left(\cdot,y_r^{*}\right)$ gives $G^e_{\mathrm{flat}}=G^i_{\mathrm{flat}}=\frac{2}{1+k_1}\Gamma\left(\cdot,y_r\right)$, and a direct computation gives $\partial_{\nu}G^e_{\mathrm{flat}}=k_1\,\partial_{\nu}G^i_{\mathrm{flat}}$ with $\nu=e_{x_2}$. The transmission conditions \eqref{transmission relation} are the standard interface conditions of the conductivity problem, and $M\nabla G^e=\nabla G^i$ follows from them together with $M\tau=\tau$ and $M\nu=\frac1{k_1}\nu$.

        It remains to control the difference
        \[
            R_r:=G-G_{\mathrm{flat}}
        \]
        on the inclusion side. Choose $R:=r_1=\frac1{16}\min\left\{\rho_0\sin\alpha_0,d_*,d_0\right\}$; then $R<\rho_0/4$, and $B_{2R}\left(y\right)$ is free of vertices and meets $\partial\mathcal P_1$ only in the side containing $y$: non-adjacent sides stay at distance at least $d_*\ge2R$ (Lemma \ref{edge separation}), the adjacent sides meet the side of $y$ only at its endpoints, which are at distance at least $\rho_0>2R$ from $y$, and $y$ stays at distance at least $\rho_0\sin\alpha_0>2R$ from the adjacent sides; in addition $B_{2R}\left(y\right)\subset\Omega$ because $\mathrm{dist}\left(y,\partial\Omega\right)\ge d_0>2R$. Take $0<r<R$. Since $G$ and $G_{\mathrm{flat}}$ have the same pole $y_r$, the same exterior coefficient $1$ and the same transmission conditions, the singularity cancels completely:
        \begin{equation}\label{eq:Rr homogeneous}
            -\operatorname{div}\left(\gamma\nabla R_r\right)=0\qquad\text{in }B_{2R}\left(y\right) .
        \end{equation}
        On $\partial B_{2R}\left(y\right)$ the pole $y_r$ and its reflection $y_r^{*}$ stay at distance at least $R$ from the circle, so \eqref{global behavior} and the explicit form of $G_{\mathrm{flat}}$ give $\left\lVert R_r\right\rVert_{L^\infty\left(\partial B_{2R}\left(y\right)\right)}\le C$ uniformly in $r\in\left(0,R\right)$. By the maximum principle for the transmission equation,
        \[
            \left\lVert R_r\right\rVert_{L^\infty\left(B_{2R}\left(y\right)\right)}\le C .
        \]
        To obtain a uniform gradient bound up to the interface, we use reflection across the flat interface to reduce the transmission problem to harmonic equations. Write $R_r^+$ and $R_r^-$ for the restrictions of $R_r$ to $B_{2R}\left(y\right)\cap\left\{x_2>0\right\}$ and $B_{2R}\left(y\right)\cap\left\{x_2<0\right\}$, respectively. On each open half-ball the coefficient $\gamma$ is constant, so $R_r^+$ and $R_r^-$ are harmonic there. Since $R_r$ is a weak solution of the transmission problem \eqref{eq:Rr homogeneous} in $B_{2R}\left(y\right)$, it satisfies on the interface segment $B_{2R}\left(y\right)\cap\left\{x_2=0\right\}$ the transmission conditions
        \[
            R_r^+=R_r^-,\qquad \partial_2R_r^+=k_1\,\partial_2R_r^- .
        \]
        Define $\widetilde R_r^-\left(x_1,x_2\right):=R_r^-\left(x_1,-x_2\right)$ on $\left\{x_2>0\right\}$ and set
        \[
            S:=R_r^++k_1\,\widetilde R_r^-,\qquad D:=R_r^+-\widetilde R_r^- .
        \]
        Then $S$ and $D$ are harmonic in $B_{2R}\left(y\right)\cap\left\{x_2>0\right\}$, and on $\left\{x_2=0\right\}$ one has $\partial_2S=\partial_2R_r^+-k_1\partial_2R_r^-=0$ and $D=0$. By the even, respectively odd, reflection principle across $\left\{x_2=0\right\}$, the even reflection of $S$ and the odd reflection of $D$ are weak solutions of $-\Delta u=0$ in the whole ball $B_{2R}\left(y\right)$, hence harmonic there by Weyl's lemma. Since $\left\lVert R_r\right\rVert_{L^\infty\left(B_{2R}\left(y\right)\right)}\le C$ and $k_1\le\lambda_0$, these reflected functions are bounded by $C$ on $B_{2R}\left(y\right)$, uniformly in $r$. The interior gradient estimate for harmonic functions therefore gives
        \[
            \left\lVert\nabla S\right\rVert_{L^\infty\left(B_{3R/2}\left(y\right)\cap\left\{x_2>0\right\}\right)}+\left\lVert\nabla D\right\rVert_{L^\infty\left(B_{3R/2}\left(y\right)\cap\left\{x_2>0\right\}\right)}\le\frac{C}{R}\le C ,
        \]
        uniformly in $r$, where the last inequality uses that $R$ is bounded below by a constant depending only on $PI$. From $R_r^+=\frac{S+k_1D}{1+k_1}$, $\widetilde R_r^-=\frac{S-D}{1+k_1}$ and $R_r^-\left(x_1,x_2\right)=\widetilde R_r^-\left(x_1,-x_2\right)$ for $x_2<0$ we obtain
        \[
            \left\lVert\nabla R_r\right\rVert_{L^\infty\left(B_R\left(y\right)\cap\overline{\mathcal P_1}\right)}+\left\lVert\nabla R_r\right\rVert_{L^\infty\left(B_R\left(y\right)\setminus\overline{\mathcal P_1}\right)}\le C ,
        \]
        uniformly in $r\in\left(0,R\right)$. Restricting to the inclusion side and using that $\nabla G^i_{\mathrm{flat}}=\frac{2}{1+k_1}\nabla\Gamma\left(\cdot,y_r\right)$ there, we obtain \eqref{local behavior function} and \eqref{local behavior gradient}.
    \end{proof}
    
    Let
    \begin{align}
        S\left( y,z \right)&=\left(k_{1}-1\right)\int_{\partial\mathcal{P}_{1}}\left(V_W\cdot\nu\right)\nabla G^e\left(\cdot,y\right)\cdot\left(M\nabla G^e\left(\cdot,z\right)\right)\mathrm ds 
        \\
        &+\kappa\int_{\mathcal{P}_{1}}\nabla G^i\left(x,y\right)\cdot\nabla G^i\left(x,z\right)\mathrm dx ,
    \end{align}
    so that $S\left(y,z\right)=B_{\mathcal P_1,k_1}\left[W,\kappa\right]\left(G\left(\cdot,y\right),G\left(\cdot,z\right)\right)$ whenever $y,z\in\Omega_0\setminus\overline\Omega$: then $G\left(\cdot,y\right)|_\Omega$ and $G\left(\cdot,z\right)|_\Omega$ are pole-free solutions whose traces lie in $H^{1/2}\left(\partial\Omega\right)$. For poles inside $\Omega$ (used in the unique continuation step below), $S\left(y,z\right)$ always refers to the integral formula above. We use throughout the symmetry $G\left(x,y\right)=G\left(y,x\right)$ of the Green function and the fact that, for fixed $x$, the map $y\mapsto G\left(x,y\right)$ solves the homogeneous equation away from $x$, so that each integrand above is a harmonic function of each pole variable away from the poles.

    \emph{Step 1 (joint injectivity).} We prove that
    \[
        B_{\mathcal P,k}\left[W,\kappa\right]\equiv0
        \quad\Longrightarrow\quad
        W=0,\qquad\kappa=0,
    \]
    for every $\left(\mathcal P,k\right)\in\overline{\mathcal A}_N\times\mathcal K_{\mathrm{adm}}$, with the parameter classes defined in Lemma \ref{polygon compactness} and Step 3 below.

    By Lemma \ref{polygon compactness}, the closure configurations retain the prescribed lower bounds on side lengths, angles, boundary distance, and separation of non-adjacent sides, together with uniform Lipschitz charts of window size $r_0/2$. The proof of the boundary identity \eqref{eq:relationshipbetweenAhandmatrix} therefore applies to these configurations for every $k\in\mathcal K_{\mathrm{adm}}$. At fixed conductivity, this identity is obtained for a suitably scaled extension of the vertex field and then extended to arbitrary vertex directions by linearity. Together with the volume term in \eqref{eq:Bdef}, it gives the bilinear form $B_{\mathcal P,k}\left[W,\kappa\right]$ for every $\kappa\in\mathbb R$.

    Fix $\left(\mathcal P_1,k_1\right)\in\overline{\mathcal A}_N\times\mathcal K_{\mathrm{adm}}$ and suppose that $B_{\mathcal P_1,k_1}\left[W,\kappa\right]\equiv0$. Let $G$ and $S$ be defined as above for this configuration. For $y,z\in U$, the restrictions $G\left(\cdot,y\right)|_\Omega$ and $G\left(\cdot,z\right)|_\Omega$ solve the homogeneous conductivity equation in $\Omega$, and their boundary traces belong to $H^{1/2}\left(\partial\Omega\right)$. Applying the assumed identity to these traces yields
    \begin{equation}\label{eq:S zero on U}
        S\left(y,z\right)=B_{\mathcal P_1,k_1}\left[W,\kappa\right]\left(G\left(\cdot,y\right)|_{\partial\Omega},G\left(\cdot,z\right)|_{\partial\Omega}\right)=0,\qquad y,z\in U .
    \end{equation}
    For fixed $z\in U$, we prove that $y\mapsto S\left(y,z\right)$ is harmonic on $E:=\Omega_0\setminus\overline{\mathcal P_1}$ in the distribution sense. Write $P=\mathcal P_1$, $k=k_1$ for brevity, and let $\eta\in C_c^{\infty}\left(E\right)$. Define
    \[
        H_\eta\left(x\right):=\int_E G\left(x,y\right)\Delta_y\eta\left(y\right)\mathrm dy .
    \]
    Since $\operatorname{supp}\eta$ stays at positive distance from $\overline P$ and the conductivity equals $1$ on $E$, the identity $\int_{\Omega_0}G\left(x,y\right)\operatorname{div}\left(\gamma\nabla\eta\right)\left(y\right)\mathrm dy=-\eta\left(x\right)$ (the Green function of $-\operatorname{div}\left(\gamma\nabla\cdot\right)$ on $\Omega_0$ with zero Dirichlet data) gives $H_\eta\left(x\right)=-\eta\left(x\right)=0$ for every $x$ in a neighborhood of $\overline P$, together with $\nabla_xH_\eta=0$ there. On every compact set $K\Subset E$ the pole stays at positive distance from $\partial P$, and there the side-reflection estimate of Lemma \ref{local behavior lemma} (applied at fixed scale away from the vertices) together with the corner bound $\left\lvert\nabla_xG\left(x,y\right)\right\rvert\le C\left(K\right)\left(\varrho^{\mu_i-1}+1\right)$ near each vertex, with coefficients continuous in the pole variable, provides a dominating function on $\partial P$ and on $P$ that is integrable in $x$ locally uniformly in $y\in K$; Fubini's theorem is therefore applicable in both terms of the integral defining $S$, and
    \[
        \begin{aligned}
            \int_E S\left(y,z\right)\Delta_y\eta\left(y\right)\mathrm dy
            &=\left(k-1\right)\int_{\partial P}\left(V_W\cdot\nu\right)\nabla_xH_\eta^{e}\cdot\left(M\nabla_xG^e\left(\cdot,z\right)\right)\mathrm ds\\
            &\quad+\kappa\int_{P}\nabla_xH_\eta^{i}\cdot\nabla_xG^i\left(\cdot,z\right)\mathrm dx
            =0,
        \end{aligned}
    \]
    because $\nabla_xH_\eta$ vanishes on $\partial P$ and on $P$. Hence $\Delta_yS\left(\cdot,z\right)=0$ in $\mathcal D'\left(E\right)$, and by Weyl's lemma $S\left(\cdot,z\right)$ has a harmonic representative on $E$. For fixed $z\in E$, the integral defining $S\left(y,z\right)$ is continuous in $y\in E$: away from the vertices the Green-function gradient traces depend continuously on the pole, near each vertex the locally uniform corner bounds provide an integrable majorant for their products ($\mu_i>1/2$), and the volume term behaves the same way; dominated convergence gives the continuity. Consequently $S\left(\cdot,z\right)$ coincides everywhere with its harmonic representative. The domain $E=\Omega_0\setminus\overline{\mathcal P_1}$ is connected, because $\mathcal P_1$ is a simple polygon compactly contained in the ball $\Omega_0$. Since $S\left(\cdot,z\right)$ vanishes on the non-empty open set $U$, unique continuation gives $S\left(\cdot,z\right)\equiv0$ on all of $E=\Omega_0\setminus\overline{\mathcal P_1}$. Fixing now $y\in\Omega_0\setminus\overline{\mathcal P_1}$ and repeating the argument in the $z$ variable, we obtain
    \begin{equation}\label{eq:S zero everywhere}
        S\left(y,z\right)=0\qquad\forall y,z\in\Omega_0\setminus\overline{\mathcal P_1}.
    \end{equation}
    We may therefore evaluate \eqref{eq:S zero everywhere} at $y=z=y_r$. Fix a side of $\mathcal P_1$ and a point $y$ on it at distance at least $\rho_0$ from the vertices, and set $y_r=y+\nu r$ as in Lemma \ref{local behavior lemma}. For $0<r<r_1$,
    \[
        0=S\left(y_r,y_r\right)=\left(k_1-1\right)\int_{\partial\mathcal P_1}\left(V_W\cdot\nu\right)\nabla G^e\left(\cdot,y_r\right)\cdot\left(M\nabla G^e\left(\cdot,y_r\right)\right)\mathrm ds+\kappa\int_{\mathcal P_1}\left\lvert\nabla G^i\left(\cdot,y_r\right)\right\rvert^2\mathrm dx .
    \]
    By \eqref{transmission relation}, on the interface $\partial\mathcal P_1$ we have $M\nabla G^e=\nabla G^i$ and $\nabla G^e=M^{-1}\nabla G^i$, hence
    \[
        \nabla G^e\cdot\left(M\nabla G^e\right)=\nabla G^i\cdot\left(M^{-1}\nabla G^i\right)\qquad\text{on }\partial\mathcal P_1 .
    \]
    Fix $\ell\in\left(0,\frac18\min\left\{\rho_0,r_1,d_0\right\}\right)$, where $\rho_0$ and $r_1$ are the constants of Lemma \ref{local behavior lemma}. We first control the contributions to \eqref{eq:S zero everywhere} that stay away from $B_\ell\left(y\right)$. The $H^1$ estimate \eqref{global behavior energy} gives, for $r\le\ell/2$,
    \[
        \int_{\mathcal P_1\setminus B_\ell\left(y\right)}\left\lvert\nabla G^i\left(x,y_r\right)\right\rvert^{2}\mathrm dx\le\left\lVert G\left(\cdot,y_r\right)\right\rVert_{H^1\left(\Omega_0\setminus B_{\ell/2}\left(y_r\right)\right)}^{2}\le C_\ell ,
    \]
    since $\mathcal P_1\setminus B_\ell\left(y\right)\subset\Omega_0\setminus B_{\ell/2}\left(y_r\right)$ and $\ell/2<\frac12\mathrm{dist}\left(y_r,\partial\Omega_0\right)$. We claim that the same holds on the boundary:
    \begin{equation}\label{eq:away boundary estimate}
        \int_{\partial\mathcal P_1\setminus B_\ell\left(y\right)}\left\lvert\nabla G^e\left(x,y_r\right)\right\rvert^{2}\mathrm ds\le C_\ell ,
    \end{equation}
    uniformly in $r\le\ell/4$. To prove \eqref{eq:away boundary estimate}, split $\partial\mathcal P_1\setminus B_\ell\left(y\right)$ into side-interior pieces and vertex pieces. On a closed subsegment of a side at positive distance from the vertices, from $B_\ell\left(y\right)$ and from the pole $y_r$, a small ball around any point of the subsegment crosses only the straight interface; the two one-sided restrictions of $G\left(\cdot,y_r\right)$ admit harmonic extensions across the interface by even and odd reflection (Lemma \ref{local behavior lemma}, applied at fixed scale), so both gradient traces are bounded there and $\left\lvert\nabla G^e\right\rvert\le C_\ell$ on such subsegments. Around each vertex $x_i$, fix a small disk $B_{\varepsilon_0}\left(x_i\right)$ with $\varepsilon_0<\frac14\min\left\{\rho_0,\ell,d_0,d_*,d_0\sin\alpha_0\right\}$: it is disjoint from $B_\ell\left(y\right)$ because $\mathrm{dist}\left(y,x_i\right)\ge\rho_0$, it contains no other vertex, no non-adjacent side (Lemma \ref{edge separation}), and, in the case $N=3$, none of the side opposite to $x_i$ (whose distance from $x_i$ is at least $d_0\sin\alpha_0$ by the side-length and angle bounds), so inside it $\partial\mathcal P_1$ consists exactly of the two sides adjacent to $x_i$. The corner expansion for the transmission problem in this two-sector domain (see \cite{hanke2024shape}, Section 3, equations (11), (13)--(14) and (19); the expansion originates in \cite{bellout1992stability}) gives, for $x$ on the two adjacent sides,
    \[
        \left\lvert\nabla G\left(x,y_r\right)\right\rvert\le C\left(\varrho^{\mu_i-1}+1\right),\qquad \varrho=\left\lvert x-x_i\right\rvert,
    \]
    where $\mu_i\in\left(\frac12,1\right)$ is the leading singular exponent at $x_i$. The gap $\mu_i>\frac12$ follows from the characteristic equation of \cite{hanke2024shape}, equations (13)--(14): with $\beta_i:=\left\lvert\alpha_i-\pi\right\rvert\ge\alpha_0$ and $a:=\frac{\left\lvert k_1-1\right\rvert}{k_1+1}\in\left(0,1\right)$, the exponent solves $\sin\left(\pi\mu\right)=a\sin\left(\beta_i\mu\right)$; for every $0<s\le\frac{\pi}{\pi+\beta_i}$ one has
    \[
        \sin\left(\pi s\right)-\sin\left(\beta_i s\right)=2\cos\frac{\left(\pi+\beta_i\right)s}{2}\sin\frac{\left(\pi-\beta_i\right)s}{2}\ge0>\left(a-1\right)\sin\left(\beta_i s\right),
    \]
    so no solution of the characteristic equation lies in that range, and $\mu_i>\frac{\pi}{\pi+\beta_i}\ge\frac{\pi}{2\pi-\alpha_0}>\frac12$. The constant $C$ is uniform in $r$: the pole $y_r$ stays at distance at least $\rho_0/2$ from $x_i$, so the data controlling the expansion, the $L^\infty$ and $H^1$ norms of $G\left(\cdot,y_r\right)$ on the ring $B_{\varepsilon_0}\left(x_i\right)\setminus B_{\varepsilon_0/2}\left(x_i\right)$, are uniformly bounded through \eqref{global behavior} and \eqref{global behavior energy}. Throughout this argument, $\left(\mathcal P_1,k_1\right)$, $y$, and $\ell$ are fixed, and the estimates are uniform in $r$. Step 3 provides the uniform lower bound over the admissible class. Since $\mu_i>\frac12$,
    \[
        \int_{\partial\mathcal P_1\cap B_{\varepsilon_0}\left(x_i\right)}\left\lvert\nabla G^e\right\rvert^{2}\mathrm ds\le C\int_0^{\varepsilon_0}\varrho^{2\mu_i-2}\mathrm d\varrho+C\le C .
    \]
    Summing the bounded contribution of the finitely many side-interior pieces and vertex pieces gives \eqref{eq:away boundary estimate}.

    Splitting the boundary integral in \eqref{eq:S zero everywhere} over $\partial\mathcal P_1\cap B_\ell\left(y\right)$ and its complement, adding and subtracting $\frac{2}{k_1+1}\Gamma\left(\cdot,y_r\right)$ and $V_W\left(y\right)$, and using the local gradient estimate \eqref{local behavior gradient} on $B_{r_1}\left(y\right)\cap\overline{\mathcal P_1}$ (which contains the side portion $B_\ell\left(y\right)\cap\partial\mathcal P_1$), we get
    \[
        0=\left(k_1-1\right)\left(V_W\left(y\right)\cdot\nu\right)b^{2}\int_{\partial\mathcal P_1\cap B_\ell\left(y\right)}\nabla\Gamma\left(\cdot,y_r\right)\cdot\left(M^{-1}\nabla\Gamma\left(\cdot,y_r\right)\right)\mathrm ds+O\left(\left\lvert\log r\right\rvert\right),\qquad b:=\frac{2}{k_1+1} ,
    \]
    where the $O\left(\left\lvert\log r\right\rvert\right)$ term collects the following contributions, all uniform in $r<\ell/4$. On $\partial\mathcal P_1\cap B_\ell\left(y\right)$ the identity \eqref{transmission relation} gives $\nabla G^e\cdot\left(M\nabla G^e\right)=\nabla G^i\cdot\left(M^{-1}\nabla G^i\right)$, and by \eqref{local behavior gradient} one has $\nabla G^i=b\nabla\Gamma+\nabla R_r$ with $\left\lvert\nabla R_r\right\rvert\le C$; the cross terms in $\nabla\Gamma\cdot\left(M^{-1}\nabla R_r\right)$ contribute $O\left(\left\lvert\log r\right\rvert\right)$, since $\int_0^\ell\left(s^{2}+r^{2}\right)^{-1/2}\mathrm ds\le C\left\lvert\log r\right\rvert$, and the affine correction $\left\lvert V_W\left(x\right)-V_W\left(y\right)\right\rvert\le C\left\lvert x-y\right\rvert$ contributes $O\left(\left\lvert\log r\right\rvert\right)$, since $\int_0^\ell s\left(s^{2}+r^{2}\right)^{-1}\mathrm ds\le C\left\lvert\log r\right\rvert$. The local part of the conductivity term satisfies $\kappa\int_{\mathcal P_1\cap B_\ell\left(y\right)}\left\lvert\nabla G^i\right\rvert^{2}\mathrm dx=O\left(\left\lvert\log r\right\rvert\right)$, since on $B_\ell\left(y\right)\cap\overline{\mathcal P_1}\subset B_{r_1}\left(y\right)\cap\overline{\mathcal P_1}$ the bound \eqref{local behavior gradient} gives $\left\lvert\nabla G^i\left(x,y_r\right)\right\rvert\le C\left\lvert x-y_r\right\rvert^{-1}+C$ and $\mathrm{dist}\left(\mathcal P_1,y_r\right)=r$. The complement of $B_\ell\left(y\right)$ contributes $O\left(1\right)$: the boundary part is bounded, in absolute value, by $\left\lvert k_1-1\right\rvert\left\lVert V_W\right\rVert_{L^\infty\left(\partial\mathcal P_1\right)}\left\lVert M\right\rVert_{L^\infty\left(\partial\mathcal P_1\right)}\int_{\partial\mathcal P_1\setminus B_\ell\left(y\right)}\left\lvert\nabla G^e\right\rvert^{2}\mathrm ds\le C_\ell$ by \eqref{eq:away boundary estimate}, and the volume part by the $H^1$ bound above. Since $M^{-1}$ is symmetric positive definite with eigenvalues $1$ and $k_1$, writing $x=y+s\tau$ on the side so that $\left\lvert\nabla\Gamma\left(x,y_r\right)\right\rvert^{2}=\left(4\pi^{2}\left(s^{2}+r^{2}\right)\right)^{-1}$ and $\nabla\Gamma\cdot\left(M^{-1}\nabla\Gamma\right)=\left(s^{2}+k_1r^{2}\right)/\left(4\pi^{2}\left(s^{2}+r^{2}\right)^{2}\right)$, the substitution $s=rt$ together with $\int_{\mathbb R}\frac{t^{2}}{\left(1+t^{2}\right)^{2}}\mathrm dt=\int_{\mathbb R}\frac{\mathrm dt}{\left(1+t^{2}\right)^{2}}=\frac{\pi}{2}$ gives
    \[
        \int_{\partial\mathcal P_1\cap B_\ell\left(y\right)}\nabla\Gamma\left(\cdot,y_r\right)\cdot\left(M^{-1}\nabla\Gamma\left(\cdot,y_r\right)\right)\mathrm ds=\frac{1+k_1}{8\pi r}+O\left(1\right)\qquad\left(r\to0\right).
    \]
    Inserting this into the decomposition of $S\left(y_r,y_r\right)$ above and multiplying by $r$, we obtain
    \[
        r\,S\left(y_r,y_r\right)=\frac{k_1-1}{2\pi\left(k_1+1\right)}\,V_W\left(y\right)\cdot\nu+O\left(r\left(1+\left\lvert\log r\right\rvert\right)\right),
    \]
    and letting $r\to0$ forces $\left(k_1-1\right)\left(V_W\left(y\right)\cdot\nu\right)=0$, hence $V_W\left(y\right)\cdot\nu=0$ because $k_1\neq1$. The point $y$ was arbitrary among the points of the sides at distance at least $\rho_0$ from the vertices, and Lemma \ref{local behavior lemma} gives $\rho_0\le d_0/4<d_0/2$, so this set contains a non-empty open interval on every side. On the side containing $y$, the function $s\mapsto V_W\left(y\left(s\right)\right)\cdot\nu$ is affine in the arc-length parameter (the field $V_W$ is affine on the side and the normal $\nu$ is constant there), and it vanishes on that open interval; hence it vanishes identically on the whole side, and in particular $W_i\cdot\nu_i=W_{i+1}\cdot\nu_i=0$ on side $i$. Applying this to the two sides adjacent to each vertex $x_i$ and using that the two outward normals are not parallel (by the angle condition $\left\lvert\alpha_i-\pi\right\rvert\ge\alpha_0$ of Definition \ref{def1}), we obtain $W_i=0$ for every $i$, that is $W=0$ (and the shape term in \eqref{eq:Bdef} vanishes). Inserting $W=0$ into \eqref{eq:S zero on U} and taking $y=z\in U$, we get
    \[
        0=S\left(y,y\right)=\kappa\int_{\mathcal P_1}\left\lvert\nabla G^i\left(\cdot,y\right)\right\rvert^2\mathrm dx .
    \]
    We show that this integral is strictly positive. If it vanished, then $\nabla G^i\left(\cdot,y\right)\equiv0$ on $\mathcal P_1$, hence $G^i\equiv C$ on $\mathcal P_1$ for some constant $C$. By \eqref{transmission relation}, on the interior segment of any side of $\mathcal P_1$ the exterior trace satisfies $G^e=C$ and $\partial_{\nu}G^e=k_1\,\partial_{\nu}G^i=0$; in other words, the harmonic function $G^e-C$ has zero Cauchy data on a non-empty relatively open segment of a straight part of its boundary. By local uniqueness for the Cauchy problem (reflection across the flat segment), $G^e\equiv C$ in an open exterior neighborhood of that segment. By unique continuation applied on the punctured exterior domain $\left(\Omega_0\setminus\overline{\mathcal P_1}\right)\setminus\left\{y\right\}$ (which is connected), we obtain $G^e\equiv C$ on $\left(\Omega_0\setminus\overline{\mathcal P_1}\right)\setminus\left\{y\right\}$. This is impossible: near the pole, $G^e\left(x,y\right)=-\frac1{2\pi}\log\left\lvert x-y\right\rvert+O\left(1\right)$ as $x\to y$ (the exterior conductivity equals $1$ near $y\in U$). Hence the integral is strictly positive, and $\kappa=0$. This proves the joint injectivity.

    \begin{lemma}\label{continuous extension operator}
        For each fixed $N$, the following holds with a constant $C=C\left(PI\right)$. For $\mathcal P\in\overline{\mathcal A}_N$ (see Step 3 below) and $W=\left(W_1,\dots,W_N\right)\in\left(\mathbb R^2\right)^N$ let $V_W$ be the continuous piecewise-affine vector field on $\partial\mathcal P$ defined by $V_W\left(x_i\right)=W_i$. The construction of the proof of Lemma \ref{extension_of_affine_map} applies verbatim to every polygon of the closure class (all its inputs — the side length bound $d_0$, the separation $d_{*}$, the graph charts with constants $\left(\frac{r_0}{2},K_0\right)$ provided by Lemma \ref{polygon compactness}, and the width $\eta=d_0/4$ of the truncation band — are uniform over the class), producing linear operators
        \[
            \mathcal E_{\mathcal P}:W\longmapsto H_{\mathcal P,W}:=\sum_{i=1}^{N}W_i\,\widehat\lambda_i^{\mathcal P},
        \]
        where $\widehat\lambda_i^{\mathcal P}\in W_c^{1,\infty}\left(\Omega\right)$ are the Lipschitz hats of that proof. These operators satisfy the following two properties.
        \begin{enumerate}
            \item For every $\mathcal P\in\overline{\mathcal A}_N$ and every $W\in\left(\mathbb R^2\right)^N$, $H_{\mathcal P,W}=\mathcal E_{\mathcal P}\left(W\right)\in W_c^{1,\infty}\left(\Omega;\mathbb R^2\right)$ satisfies $H_{\mathcal P,W}|_{\partial\mathcal P}=V_W$ and $\left\lVert H_{\mathcal P,W}\right\rVert_{W^{1,\infty}\left(\Omega\right)}\le C\sum_{i=1}^{N}\left\lvert W_i\right\rvert$.
            \item Whenever $\mathcal P_n,\mathcal P\in\overline{\mathcal A}_N$ satisfy $\mathcal P_n\to\mathcal P$ in vertex coordinates and $W^{\left(n\right)}\to W$, one can choose extensions $H_{\mathcal P_n,W^{\left(n\right)}}$ and $H_{\mathcal P,W}$ as in item 1 and global bi-Lipschitz homeomorphisms $\Psi_n=\mathrm{Id}+g_n:\Omega\to\Omega$ such that $\Psi_n\left(\mathcal P\right)=\mathcal P_n$, $\Psi_n=\mathrm{Id}$ near $\partial\Omega$, $\left\lVert g_n\right\rVert_{W^{1,\infty}}\le C\max_i\left\lvert x_i^n-x_i\right\rvert\to0$, and such that
            \[
                H_{\mathcal P_n,W^{\left(n\right)}}\circ\Psi_n\longrightarrow H_{\mathcal P,W}\qquad\text{in }W^{1,\infty}\left(\Omega\right) .
            \]
        \end{enumerate}
    \end{lemma}
    \begin{proof}
        Item 1. The hats $\widehat\lambda_i^{\mathcal P}$ of the proof of Lemma \ref{extension_of_affine_map} are obtained from the boundary hat functions $\lambda_i$ (affine on the sides adjacent to $x_i$, vanishing on the others), the uniform boundary Lipschitz constant $L_0=L_0\left(PI\right)$, and the McShane extension truncated on the band $\left\{x:\ \mathrm{dist}\left(x,\partial\mathcal P\right)<d_0/4\right\}$, which is compactly contained in $\Omega$ because $\mathrm{dist}\left(\mathcal P,\partial\Omega\right)\ge d_0$. Every input of this construction is controlled by $PI$ alone, uniformly over the whole class $\overline{\mathcal A}_N$; hence $\left\lVert\widehat\lambda_i^{\mathcal P}\right\rVert_{W^{1,\infty}\left(\Omega\right)}\le\max\left\{1,L_0\right\}$, $H_{\mathcal P,W}=\sum_iW_i\widehat\lambda_i^{\mathcal P}$ vanishes near $\partial\Omega$, and $H_{\mathcal P,W}|_{\partial\mathcal P}=\sum_iW_i\lambda_i=V_W$, together with $\left\lVert H_{\mathcal P,W}\right\rVert_{W^{1,\infty}}\le C\sum_i\left\lvert W_i\right\rvert$, $C=C\left(PI\right)$. This proves item 1.

        Item 2. Fix the limit polygon $\mathcal P$ and set $\delta_i^n:=x_i^n-x_i$, $g_n:=\mathcal E_{\mathcal P}\left(\delta^n\right)=\sum_i\delta_i^n\,\widehat\lambda_i^{\mathcal P}$ and $\Psi_n:=\mathrm{Id}+g_n$. Item 1 gives $\left\lVert g_n\right\rVert_{W^{1,\infty}}\le C\sum_i\left\lvert x_i^n-x_i\right\rvert\to0$, so for $n$ large $\Psi_n$ is a bi-Lipschitz homeomorphism of $\Omega$ onto itself, equal to the identity near $\partial\Omega$, with $\left\lVert D\Psi_n-I\right\rVert_{L^\infty}\le C\sum_i\left\lvert x_i^n-x_i\right\rvert$. On $\partial\mathcal P$ the field $g_n$ is the piecewise-affine vertex field with values $\delta_i^n$; along each side one has $\Psi_n\left(\left(1-t\right)x_i+tx_{i+1}\right)=\left(1-t\right)x_i^{\left(n\right)}+tx_{i+1}^{\left(n\right)}$, so the side is mapped affinely onto the corresponding side of $\mathcal P_n$, and $\Psi_n\left(\partial\mathcal P\right)=\partial\mathcal P_n$. Being the identity near $\partial\Omega$, $\Psi_n$ maps the component of $\Omega\setminus\partial\mathcal P$ bounded by $\partial\mathcal P$, namely $\mathcal P$, onto the component of $\Omega\setminus\partial\mathcal P_n$ bounded by $\partial\mathcal P_n$, namely $\mathcal P_n$. Take $H_{\mathcal P,W}:=\mathcal E_{\mathcal P}\left(W\right)$ and $H_{\mathcal P_n,W^{\left(n\right)}}:=\mathcal E_{\mathcal P}\left(W^{\left(n\right)}\right)\circ\Psi_n^{-1}$, always through the operator of the reference polygon $\mathcal P$. For $x=\left(1-t\right)x_i^{\left(n\right)}+tx_{i+1}^{\left(n\right)}\in\partial\mathcal P_n$,
        \[
            H_{\mathcal P_n,W^{\left(n\right)}}\left(x\right)=\mathcal E_{\mathcal P}\left(W^{\left(n\right)}\right)\left(\left(1-t\right)x_i+tx_{i+1}\right)=\left(1-t\right)W_i^{\left(n\right)}+tW_{i+1}^{\left(n\right)}=V_{W^{\left(n\right)}}\left(x\right),
        \]
        so the trace condition holds; moreover $\left\lVert H_{\mathcal P_n,W^{\left(n\right)}}\right\rVert_{W^{1,\infty}}\le2C\sum_i\left\lvert W_i^{\left(n\right)}\right\rvert$ for $n$ large, since $\left\lVert D\Psi_n^{-1}\right\rVert_{L^\infty}\le2$, so item 1 applies along the sequence with the same constant. Finally,
        \[
            H_{\mathcal P_n,W^{\left(n\right)}}\circ\Psi_n=\mathcal E_{\mathcal P}\left(W^{\left(n\right)}\right)\longrightarrow\mathcal E_{\mathcal P}\left(W\right)=H_{\mathcal P,W}\qquad\text{in }W^{1,\infty}\left(\Omega\right),
        \]
        because $\left\lVert\mathcal E_{\mathcal P}\left(W^{\left(n\right)}\right)-\mathcal E_{\mathcal P}\left(W\right)\right\rVert_{W^{1,\infty}}\le C\sum_i\left\lvert W_i^{\left(n\right)}-W_i\right\rvert\to0$ by linearity. This proves item 2.
    \end{proof}

    \begin{lemma}\label{polygon compactness}
        For each fixed $N\in\left\{3,\dots,N_0\right\}$, the closure $\overline{\mathcal A}_N$, taken in vertex coordinates, of the admissible polygons with exactly $N$ vertices is compact. Every element of $\overline{\mathcal A}_N$ is a simple polygon (so that $\mathcal P=\overline{\operatorname{int}\mathcal P}$) whose interior satisfies the graph condition of Definition \ref{def1} with the window size $\frac{r_0}{2}$ and the same slope bound $K_0$ in place of $\left(r_0,K_0\right)$; the side lengths at least $d_0$, the angle bounds $\alpha_0\le\alpha_i\le2\pi-\alpha_0$ with $\left\lvert\alpha_i-\pi\right\rvert\ge\alpha_0$, the distance $\mathrm{dist}\left(\mathcal P,\partial\Omega\right)\ge d_0$ and the separation $\mathrm{dist}\left(e_i,e_j\right)\ge d_{*}$ of non-adjacent sides hold on $\overline{\mathcal A}_N$ with the original constants.
    \end{lemma}
    \begin{proof}
        The vertex coordinates lie in $\overline\Omega^{N}$, so the closure is bounded and compact. The algebraic conditions (side lengths, angles, distance to $\partial\Omega$) are closed in vertex coordinates; the distance between two non-degenerate segments depends continuously on their endpoints, so $\mathrm{dist}\left(e_i,e_j\right)\ge d_{*}$ passes to the limit for non-adjacent sides. Consequently every element of $\overline{\mathcal A}_N$ has sides of length at least $d_0$, non-degenerate angles in $\left[\alpha_0,2\pi-\alpha_0\right]$ with $\left\lvert\alpha_i-\pi\right\rvert\ge\alpha_0$, and pairwise disjoint non-adjacent sides; adjacent sides share exactly their common endpoint, so the closed polygonal line is simple, and the limit object is a simple polygon $\mathcal P=\overline{\operatorname{int}\mathcal P}$. It remains to prove the graph condition. Take $p\in\partial\mathcal P$ and admissible $\mathcal P_n$ with exactly $N$ vertices such that $\mathcal P_n\to\mathcal P$ in vertex coordinates; choose $p_n\in\partial\mathcal P_n$ with $p_n\to p$. In the chart of Definition \ref{def1} centered at $p_n$, the boundary is the graph of $\varphi_n$ with $\varphi_n\left(0\right)=0$ and $\operatorname{Lip}\varphi_n\le K_0$ on $\left[-r_0,r_0\right]$. Passing to a subsequence, the chart rotations converge; the $\varphi_n$ are equi-Lipschitz, and $\partial\mathcal P_n\to\partial\mathcal P$ in the Hausdorff metric (the vertices converge and the sides are straight segments), so $\varphi_n\to\varphi$ uniformly on $\left[-r_0,r_0\right]$, with $\varphi\left(0\right)=0$ and $\operatorname{Lip}\varphi\le K_0$. Consider the smaller rectangle $Q':=\left[-\frac{r_0}{2},\frac{r_0}{2}\right]\times\left[-\frac{K_0r_0}{2},\frac{K_0r_0}{2}\right]$ in the limit chart: it sits inside the full chart rectangles with a uniform margin. A point of $Q'$ lying strictly above the graph of $\varphi$ stays at a fixed margin from it, hence lies in $\operatorname{int}\mathcal P_n$ for all $n$ large and passes to $\operatorname{int}\mathcal P$; a point of $Q'$ strictly below the graph stays at a fixed margin from $\overline{\mathcal P_n}$, hence is not in $\overline{\mathcal P}$; and the graph itself consists of limits of points of $\partial\mathcal P_n$, so it is contained in $\partial\mathcal P$. Therefore
        \[
            \left(\operatorname{int}\mathcal P\right)\cap Q'=\left\{\left(x_1,x_2\right)\in Q':\ x_2>\varphi\left(x_1\right)\right\},
            \qquad
            \partial\mathcal P\cap Q'=\left\{\left(x_1,\varphi\left(x_1\right)\right):\ \left\lvert x_1\right\rvert\le\frac{r_0}{2}\right\},
        \]
        which is the graph condition of Definition \ref{def1} with window size $\frac{r_0}{2}$ and slope bound $K_0$. The proof is complete.
    \end{proof}

    \emph{Step 2 (operator-norm continuity via the distributed formula).} We prove operator-norm continuity using the distributed representation of $B_{\mathcal P,k}\left[W,\kappa\right]$. After pulling the configurations back to a fixed domain, this representation allows us to combine $L^\infty$ convergence of the coefficient matrices with uniform energy estimates for the solutions. Given $\left(\mathcal P,k\right)\in\overline{\mathcal A}_N\times\mathcal K_{\mathrm{adm}}$ and $W\in\left(\mathbb R^2\right)^N$, let $H_{\mathcal P,W}$ be an extension of $V_W$ provided by Lemma \ref{continuous extension operator}, item 1, and set
    \[
        \mathcal A_{\mathcal P,W}:=\left(\mathrm{div}\,H_{\mathcal P,W}\right)I-DH_{\mathcal P,W}-\left(DH_{\mathcal P,W}\right)^T .
    \]
    Since $H_{\mathcal P,W}=V_W$ on $\partial\mathcal P$, the boundary identity of Theorem \ref{mainderivativestheorem} gives
    \begin{equation}\label{distributed B}
        \left(k-1\right)\int_{\partial\mathcal P}\left(V_W\cdot\nu\right)\nabla u^\phi\cdot\left(M\nabla u^\psi\right)\mathrm ds=\int_\Omega\gamma_{\mathcal P}^k\,\nabla u^\phi\cdot\left(\mathcal A_{\mathcal P,W}\nabla u^\psi\right)\mathrm dx ,
    \end{equation}
    Applying Theorem \ref{mainderivativestheorem} to $\varepsilon H_{\mathcal P,W}$, with $\varepsilon>0$ chosen so that $\varepsilon\left\lVert DH_{\mathcal P,W}\right\rVert_{L^\infty}\le\frac12$, and dividing by $\varepsilon$, we obtain \eqref{distributed B} for every extension $H_{\mathcal P,W}$. Therefore
    \begin{equation}\label{distributed B full}
        B_{\mathcal P,k}\left[W,\kappa\right]\left(\phi,\psi\right)=\int_\Omega\gamma_{\mathcal P}^k\,\nabla u^\phi\cdot\left(\mathcal A_{\mathcal P,W}\nabla u^\psi\right)\mathrm dx+\kappa\int_{\mathcal P}\nabla u^\phi\cdot\nabla u^\psi\,\mathrm dx ,
    \end{equation}
    which involves no gradient traces on $\partial\mathcal P$. We claim that the map $\left(\mathcal P,k,W,\kappa\right)\mapsto B_{\mathcal P,k}\left[W,\kappa\right]$ is continuous into the space of bounded bilinear forms on $H^{1/2}\left(\partial\Omega\right)\times H^{1/2}\left(\partial\Omega\right)$ equipped with the operator norm. Let $\mathcal P_n\to\mathcal P$ in vertex coordinates (same $N$), $k_n\to k$, $W^{\left(n\right)}\to W$, $\kappa_n\to\kappa$, and let $\Psi_n=\mathrm{Id}+g_n$ and the extensions $H_{\mathcal P_n,W^{\left(n\right)}}$, $H_{\mathcal P,W}$ be the compatible choices given by Lemma \ref{continuous extension operator}, item 2, so that $\Psi_n\left(\mathcal P\right)=\mathcal P_n$ (hence $\gamma_{\mathcal P_n}^{k_n}\circ\Psi_n=1+\left(k_n-1\right)\chi_{\mathcal P}$), $\Psi_n=\mathrm{Id}$ near $\partial\Omega$, $\left\lVert\Psi_n-\mathrm{Id}\right\rVert_{W^{1,\infty}}\to0$ and $H_{\mathcal P_n,W^{\left(n\right)}}\circ\Psi_n\to H_{\mathcal P,W}$ in $W^{1,\infty}\left(\Omega\right)$. Let $u_n^f$ and $u^f$ be the solutions for the conductivities $\gamma_{\mathcal P_n}^{k_n}$ and $\gamma_{\mathcal P}^k$ with the same datum $f$, and set $U_n^f:=u_n^f\circ\Psi_n$. The pull-back energy estimates of Section 4 give
    \[
        \sup_{\left\lVert f\right\rVert_{H^{1/2}}\le1}\left\lVert\nabla\left(U_n^f-u^f\right)\right\rVert_{L^2\left(\Omega\right)}\to0 .
    \]
    Pulling \eqref{distributed B full} back to the reference configuration $\mathcal P$: the transported matrices converge in $L^\infty$, since by the chain rule $\left(DH_{\mathcal P_n,W^{\left(n\right)}}\right)\circ\Psi_n=D\left(H_{\mathcal P_n,W^{\left(n\right)}}\circ\Psi_n\right)\left(D\Psi_n\right)^{-1}\to DH_{\mathcal P,W}$ and likewise for the divergence, while $k_n\to k$ and $D\Psi_n\to I$ uniformly; together with $H_{\mathcal P_n,W^{\left(n\right)}}\circ\Psi_n\to H_{\mathcal P,W}$ in $W^{1,\infty}\left(\Omega\right)$ (Lemma \ref{continuous extension operator}), a Cauchy--Schwarz computation gives
    \[
        \sup_{\left\lVert\phi\right\rVert_{H^{1/2}},\left\lVert\psi\right\rVert_{H^{1/2}}\le1}\left\lvert B_{\mathcal P_n,k_n}\left[W^{\left(n\right)},\kappa_n\right]\left(\phi,\psi\right)-B_{\mathcal P,k}\left[W,\kappa\right]\left(\phi,\psi\right)\right\rvert\to0 ,
    \]
    i.e.\ $\left\lVert B_{\mathcal P_n,k_n}\left[W^{\left(n\right)},\kappa_n\right]-B_{\mathcal P,k}\left[W,\kappa\right]\right\rVert_*\to0$, which proves the claimed continuity.

    \emph{Step 3 (compactness and conclusion).} For the compactness argument we work with the full class of polygons satisfying the prescribed prior conditions: any admissible family is a subset of that class, so the lower bound obtained below applies uniformly to every admissible family. We use the closed parameter classes: $\overline{\mathcal A}_N$, $N=3,\dots,N_0$, as in Lemma \ref{polygon compactness} (the vertex-coordinate closure of the polygons with exactly $N$ vertices in this full class), and
    \[
        \mathcal K_{\mathrm{adm}}:=\left\{k:\ \frac1{\lambda_0}\le k\le\lambda_0,\quad \frac1{\lambda_1}\le\left\lvert k-1\right\rvert\le\lambda_1\right\},
    \]
    the compact admissible set containing $\mathcal K$; in particular $k\neq1$ throughout the compact set $\mathcal K_{\mathrm{adm}}$, so the shape direction does not degenerate in the limit. The estimate obtained below then holds for the original classes as well. For each $N\in\left\{3,\dots,N_0\right\}$ consider the tangent sphere
    \[
        \mathfrak X_N=\left\{\left(\mathcal P,k,W,\kappa\right):\ \mathcal P\in\overline{\mathcal A}_N,\ k\in\mathcal K_{\mathrm{adm}},\ \sum_{i=1}^{N}\left\lvert W_i\right\rvert+\left\lvert\kappa\right\rvert=1\right\},
    \]
    and set $\mathfrak X=\bigcup_{N=3}^{N_0}\mathfrak X_N$. By Lemma \ref{polygon compactness}, $\overline{\mathcal A}_N$ is compact for each fixed $N$; the direction sphere $\left\{\left(W,\kappa\right):\sum_i\left\lvert W_i\right\rvert+\left\lvert\kappa\right\rvert=1\right\}$ and the conductivity set $\mathcal K_{\mathrm{adm}}$ are also compact, so each $\mathfrak X_N$ is compact, and so is the finite union $\mathfrak X$. By Step 2 the map $\left(\mathcal P,k,W,\kappa\right)\mapsto\left\lVert B_{\mathcal P,k}\left[W,\kappa\right]\right\rVert_*$ is continuous on $\mathfrak X$, and by Step 1 it is strictly positive at every point of $\mathfrak X$ (if $B_{\mathcal P,k}\left[W,\kappa\right]\equiv0$ then $W=0$ and $\kappa=0$, contradicting the normalization). Therefore
    \[
        m_1:=\inf_{\mathfrak X}\left\lVert B_{\mathcal P,k}\left[W,\kappa\right]\right\rVert_*>0,
    \]
    and $m_1$ depends only on the prior information $PI$. For any $\left(\mathcal P,k\right)\in\mathcal A\times\mathcal K$ and any $\left(W,\kappa\right)$ with $q:=\sum_i\left\lvert W_i\right\rvert+\left\lvert\kappa\right\rvert>0$, the normalized direction $\left(W/q,\kappa/q\right)$ belongs to $\mathfrak X$, and by linearity of $B$ in $\left(W,\kappa\right)$,
    \[
        \left\lVert B_{\mathcal P,k}\left[W,\kappa\right]\right\rVert_*=q\left\lVert B_{\mathcal P,k}\left[\frac{W}{q},\frac{\kappa}{q}\right]\right\rVert_*\ge m_1 q ,
    \]
    which is \eqref{lowerbound operator version}. For \eqref{lowerbound}, by the definition of the operator norm there exist $\tilde\phi,\tilde\psi\in H^{1/2}\left(\partial\Omega\right)$ with
    \[
        \left\lvert B_{\mathcal P,k}\left[W,\kappa\right]\left(\tilde\phi,\tilde\psi\right)\right\rvert\ge\frac{m_1}{2}q\left\lVert\tilde\phi\right\rVert_{H^{1/2}\left(\partial\Omega\right)}\left\lVert\tilde\psi\right\rVert_{H^{1/2}\left(\partial\Omega\right)},
    \]
    and since, for $\mathcal P=\mathcal P_1$, $k=k_1$, $W_i=x_2^i-x_1^i$ and $\kappa=k_2-k_1$, we have $B_{\mathcal P,k}\left[W,\kappa\right]\left(\tilde\phi,\tilde\psi\right)=\frac{d}{dt}F\left(t,\tilde\phi,\tilde\psi\right)\big|_{t=0}$, this is \eqref{lowerbound}. The proof is complete.
  
  \end{proof}
  
\section{Lipschitz stability result}
    In this section, we prove our main result Theorem \ref{mainthm} 
    \begin{proof}
        Let $\rho=\lVert\Lambda^{\mathcal{P}_{1}}_{k_{1}}-\Lambda^{\mathcal{P}_{2}}_{k_{2}}\rVert_{*}$. If $\rho=0$, uniqueness for the two-dimensional conductivity problem (\cite{astala2006calderon}) gives $\gamma^{\mathcal P_1}_{k_1}=\gamma^{\mathcal P_2}_{k_2}$ a.e.\ in $\Omega$, and the conclusion is immediate. We henceforth assume $0<\rho$.

        We distinguish two cases: (a) $\rho>\delta_{0}$ and (b) $\rho\leq\delta_{0}$, where $0<\delta_{0}<\min\left\{\delta,\rho_*\right\}$, with $\delta$ the constant of Theorem \ref{final_same_number_result} and $\rho_*$ that of Lemma \ref{Logarithmic_stability_estimate}, will be chosen below. In case (a), the uniform bounds on the conductivities give
        \[
            \lVert\gamma_{k_{1}}^{\mathcal{P}_{1}}-\gamma_{k_{2}}^{\mathcal{P}_{2}}\rVert_{L^{1}\left(\Omega\right)}\le\left(\lambda_{1}+\lambda_{0}-\frac1{\lambda_0}\right)\lvert\Omega\rvert\le C\frac{\rho}{\delta_{0}}\le C\rho ,
        \]
        since $\rho>\delta_0$.

        Assume from now on $0<\rho\le\delta_{0}$. By Theorem \ref{final_same_number_result} (valid because $\delta_0<\delta$ by the choice below), $\mathcal P_1$ and $\mathcal P_2$ have the same number $N$ of vertices, and we denote the corresponding vertices by $x_1^i$ and $x_2^i$, $i=1,\dots,N$. Set
        \[
            W_i=x_2^i-x_1^i,\qquad v=\sum_{i=1}^{N}\left\lvert W_i\right\rvert,\qquad \kappa=k_2-k_1,\qquad q=v+\left\lvert\kappa\right\rvert .
        \]
        By Lemma \ref{maindistancelemma}, Lemma \ref{first_distance_estimate} and Lemma \ref{same_number_result},
        \begin{equation}\label{rough v smallness}
            v\le C\left\lvert\log\rho\right\rvert^{-\alpha} .
        \end{equation}
        To control $\kappa$, we use the logarithmic estimate of Lemma \ref{Logarithmic_stability_estimate}: since $\gamma_{k_1}^{\mathcal P_1}-\gamma_{k_2}^{\mathcal P_2}=k_1-k_2=-\kappa$ on $\mathcal P_1\cap\mathcal P_2$,
        \begin{equation}\label{kappa lower bound}
            \left\lvert\kappa\right\rvert^{2}\left\lvert\mathcal P_1\cap\mathcal P_2\right\rvert\le\left\lVert\gamma_{k_{1}}^{\mathcal{P}_{1}}-\gamma_{k_{2}}^{\mathcal{P}_{2}}\right\rVert_{L^{2}\left(\Omega\right)}^{2} .
        \end{equation}
        Moreover, by Lemma \ref{area lower bound} and \eqref{estimate_P1triangleP2},
        \[
            \left\lvert\mathcal P_1\cap\mathcal P_2\right\rvert=\left\lvert\mathcal P_1\right\rvert-\left\lvert\mathcal P_1\setminus\mathcal P_2\right\rvert\ge a_*-\left\lvert\mathcal P_1\triangle\mathcal P_2\right\rvert\ge\frac{a_*}{2}
        \]
        for $\rho$ small enough, since $\left\lvert\mathcal P_1\triangle\mathcal P_2\right\rvert\le C\left\lvert\log\rho\right\rvert^{-2\alpha}\to0$. Hence
        \begin{equation}\label{rough kappa smallness}
            \left\lvert\kappa\right\rvert\le C\left\lVert\gamma_{k_{1}}^{\mathcal{P}_{1}}-\gamma_{k_{2}}^{\mathcal{P}_{2}}\right\rVert_{L^{2}\left(\Omega\right)}\le C\left\lvert\log\rho\right\rvert^{-\alpha} ,
        \end{equation}
        and altogether
        \begin{equation}\label{rough q smallness}
            q=v+\left\lvert\kappa\right\rvert\le C\left\lvert\log\rho\right\rvert^{-\alpha}\to0\qquad\left(\rho\to0\right) .
        \end{equation}

        We now choose the small-data threshold $\delta_0>0$ with $\delta_0<\min\left\{\delta,\rho_*\right\}$. Let $C_E$ be the constant of Lemma \ref{extension_of_affine_map}, so that the extension $h$ of the vertex displacement $W_i=x_2^i-x_1^i$ satisfies $\left\lVert h\right\rVert_{W^{1,\infty}}\le C_Ev\le C_Eq$, and set $C_R:=\left(C_E+1\right)^{2}C$, where $C$ is the constant of \eqref{continuity_of_derivatives}; the continuity estimate used below then reads $\left\lvert\frac{dF}{dt}\left(t\right)-\frac{dF}{dt}\left(0\right)\right\rvert\le C_R\,t\,q^{2}\left\lVert\tilde\phi\right\rVert_{H^{1/2}}\left\lVert\tilde\psi\right\rVert_{H^{1/2}}$. The constant $m_1>0$ of Theorem \ref{lowerbound result} is fixed before $\delta_0$ is chosen. The choice is made in two stages: first take $\delta_0$ small enough that, for $\rho\le\delta_0$,
        \[
            \left\lvert\mathcal P_1\triangle\mathcal P_2\right\rvert\le\frac{a_*}{2},
        \]
        which is the geometric threshold (it gives $\left\lvert\mathcal P_1\cap\mathcal P_2\right\rvert\ge a_*/2$ by Lemma \ref{area lower bound} and thereby validates \eqref{rough kappa smallness} in the regime $\rho\le\delta_0$); then shrink $\delta_0$ further, if necessary, so that additionally
        \[
            C_E\,q\le\frac12,\qquad C_R\,q\le\frac{m_1}{2}.
        \]
        Both stages are feasible: for $\rho<\rho_*$ the quantities $q\le C_0\left\lvert\log\rho\right\rvert^{-\alpha}$ and $\left\lvert\mathcal P_1\triangle\mathcal P_2\right\rvert\le C\left\lvert\log\rho\right\rvert^{-2\alpha}$ tend to $0$, with all constants already fixed. The condition $C_Eq\le1/2$ gives $\left\lVert Dh\right\rVert_{L^\infty}\le\frac12$. The bound $\left\lVert Dh\right\rVert_{L^\infty}\le1/2$ and the ellipticity of $k(t)$ ensure that Theorem \ref{mainderivativestheorem} applies along the entire path. Moreover $h$ is affine on each side of $\partial\mathcal P_1$ with the values $W_i$ at the vertices, so $\Phi_{h1}$ maps each side of $\mathcal P_1$ onto the corresponding side of $\mathcal P_2$ and, being the identity near $\partial\Omega$, $\Phi_{h1}\left(\mathcal P_1\right)=\mathcal P_2$; hence $F\left(1,\tilde\phi,\tilde\psi\right)=\left\langle\Lambda_{k_2}^{\mathcal P_2}\tilde\phi,\tilde\psi\right\rangle$ and $F\left(0,\tilde\phi,\tilde\psi\right)=\left\langle\Lambda_{k_1}^{\mathcal P_1}\tilde\phi,\tilde\psi\right\rangle$. The condition
        \begin{equation}\label{smallness of q}
            q\le\frac{m_1}{2C_R},
        \end{equation}
        makes the quadratic term in the closure below absorbable.

        By Theorem \ref{lowerbound result} there exist $\tilde\phi,\tilde\psi\in H^{1/2}\left(\partial\Omega\right)$ such that
        \[
            \left|\frac{dF}{dt}\left(t,\tilde\phi,\tilde\psi\right)\Big|_{t=0}\right|\ge\frac{m_1}{2}q\left\lVert\tilde\phi\right\rVert_{H^{1/2}}\left\lVert\tilde\psi\right\rVert_{H^{1/2}} ,
        \]
        and by Theorem \ref{mainderivativestheorem}, for every $t\in\left[0,1\right]$,
        \[
            \left|\frac{dF}{dt}\left(t,\tilde\phi,\tilde\psi\right)-\frac{dF}{dt}\left(t,\tilde\phi,\tilde\psi\right)\Big|_{t=0}\right|\le Ct\left(\left\lVert h\right\rVert_{W^{1,\infty}}+\left\lvert\kappa\right\rvert\right)^{2}\left\lVert\tilde\phi\right\rVert_{H^{1/2}}\left\lVert\tilde\psi\right\rVert_{H^{1/2}}\le C_R\,t\,q^{2}\left\lVert\tilde\phi\right\rVert_{H^{1/2}}\left\lVert\tilde\psi\right\rVert_{H^{1/2}} ,
        \]
        since $\left\lVert h\right\rVert_{W^{1,\infty}}+\left\lvert\kappa\right\rvert\le C_Ev+\left\lvert\kappa\right\rvert\le\left(C_E+1\right)q$ by the bound of Lemma \ref{extension_of_affine_map}. Therefore
        \begin{equation}\label{closure}
            \begin{aligned}
                \rho\left\lVert\tilde\phi\right\rVert_{H^{1/2}}\left\lVert\tilde\psi\right\rVert_{H^{1/2}}
                &\ge\left\lvert\left\langle\left(\Lambda^{\mathcal{P}_{1}}_{k_{1}}-\Lambda^{\mathcal{P}_{2}}_{k_{2}}\right)\tilde{\phi},\tilde{\psi}\right\rangle\right\rvert
                =\left\lvert F\left(1,\tilde{\phi},\tilde{\psi}\right)-F\left(0,\tilde{\phi},\tilde{\psi}\right)\right\rvert\\
                &\ge\left\lvert\frac{dF}{dt}\Big|_{t=0}\right\rvert-\int_{0}^{1}\left\lvert\frac{dF}{dt}\left(t\right)-\frac{dF}{dt}\Big|_{t=0}\right\rvert\mathrm dt\\
                &\ge\left(\frac{m_1}{2}q-\frac{C_R}{2}q^{2}\right)\left\lVert\tilde\phi\right\rVert_{H^{1/2}}\left\lVert\tilde\psi\right\rVert_{H^{1/2}}
                \ge\frac{m_1}{4}q\left\lVert\tilde\phi\right\rVert_{H^{1/2}}\left\lVert\tilde\psi\right\rVert_{H^{1/2}} ,
            \end{aligned}
        \end{equation}
        where the last inequality follows from \eqref{smallness of q}. Hence
        \begin{equation}\label{q lipschitz}
            q=v+\left\lvert\kappa\right\rvert\le\frac{4}{m_1}\rho .
        \end{equation}

        Since $\mathcal P_2=\Phi_{h1}\left(\mathcal P_1\right)$ and $\left\lVert h\right\rVert_{L^{\infty}}\le Cv$, every point of the symmetric difference lies within distance $Cv$ of $\partial\mathcal P_1$, and the tubular-neighborhood estimate (using $\mathcal H^{1}\left(\partial\mathcal P_1\right)\le N_0L$ and the uniform Lipschitz character of $\partial\mathcal P_1$) gives
        \begin{equation}\label{triangle estimate}
            \left\lvert\mathcal P_1\triangle\mathcal P_2\right\rvert\le Cv .
        \end{equation}
        Finally,
        \begin{equation}
            \begin{aligned}
                \lVert\gamma_{k_{1}}^{\mathcal{P}_{1}}-\gamma_{k_{2}}^{\mathcal{P}_{2}}\rVert_{L^{1}\left(\Omega\right)}
                &=\left\lvert k_1-1\right\rvert\left\lvert\mathcal P_1\setminus\mathcal P_2\right\rvert+\left\lvert k_2-1\right\rvert\left\lvert\mathcal P_2\setminus\mathcal P_1\right\rvert+\left\lvert\kappa\right\rvert\left\lvert\mathcal P_1\cap\mathcal P_2\right\rvert\\
                &\le C\left\lvert\mathcal P_1\triangle\mathcal P_2\right\rvert+C\left\lvert\kappa\right\rvert
                \le C\left(v+\left\lvert\kappa\right\rvert\right)\le C\rho ,
            \end{aligned}
        \end{equation}
        by \eqref{triangle estimate} and \eqref{q lipschitz}, which is the desired estimate.
    \end{proof}

    The preceding proof also yields Lipschitz stability for the boundary of the inclusion and for its conductivity value, separately.

    \begin{corollary}\label{joint corollary}
        Under the assumptions of Theorem \ref{mainthm} there exists a constant $C>0$, depending only on the prior information, such that for any $\mathcal P_1,\mathcal P_2\in\mathcal A$ and $k_1,k_2\in\mathcal K$,
        \[
            \mathrm{d}_H\left(\partial\mathcal P_1,\partial\mathcal P_2\right)+\left\lvert k_1-k_2\right\rvert\le C\left\lVert\Lambda_{k_1}^{\mathcal P_1}-\Lambda_{k_2}^{\mathcal P_2}\right\rVert_* .
        \]
        When the two Dirichlet-to-Neumann maps are sufficiently close, the polygons have the same number of vertices, and with the corresponding labeling
        \[
            \sum_{i=1}^{N}\left\lvert x_1^i-x_2^i\right\rvert+\left\lvert k_1-k_2\right\rvert\le C\left\lVert\Lambda_{k_1}^{\mathcal P_1}-\Lambda_{k_2}^{\mathcal P_2}\right\rVert_* .
        \]
    \end{corollary}
    \begin{proof}
        Write $\rho:=\left\lVert\Lambda_{k_1}^{\mathcal P_1}-\Lambda_{k_2}^{\mathcal P_2}\right\rVert_*$, and let $\delta_0$ be the threshold chosen in the proof of Theorem \ref{mainthm}. If $\rho=0$, uniqueness (\cite{astala2006calderon}) identifies the two conductivities and hence the two configurations. If $\rho>\delta_0$, then
        \[
            \mathrm{d}_H\left(\partial\mathcal P_1,\partial\mathcal P_2\right)+\left\lvert k_1-k_2\right\rvert
            \le L+\lambda_0-\lambda_0^{-1}
            \le\frac{L+\lambda_0-\lambda_0^{-1}}{\delta_0}\,\rho ,
        \]
        which proves the first estimate in this regime. If $0<\rho\le\delta_0$, the polygons have the same number of vertices, and \eqref{q lipschitz} gives
        \[
            q:=\sum_{i=1}^{N}\left\lvert x_1^i-x_2^i\right\rvert+\left\lvert k_1-k_2\right\rvert
            \le\frac{4}{m_1}\rho .
        \]
        The affine correspondence between the sides implies
        \[
            \mathrm{d}_H\left(\partial\mathcal P_1,\partial\mathcal P_2\right)\le\max_i\left\lvert x_1^i-x_2^i\right\rvert
        \]
        (each point of a side of $\mathcal P_2$ is the image of the point of the corresponding side of $\mathcal P_1$ with the same parameter, and the correspondence is invertible), hence
        \[
            \mathrm{d}_H\left(\partial\mathcal P_1,\partial\mathcal P_2\right)+\left\lvert k_1-k_2\right\rvert
            \le q\le\frac{4}{m_1}\rho .
        \]
        This proves both assertions.
    \end{proof}

\bibliographystyle{plain}
\bibliography{references}
\end{document}